\documentclass[12pt]{article}

\usepackage[a4paper, margin=1in]{geometry}
\usepackage{changepage}
\usepackage[hidelinks]{hyperref}

\usepackage{amsmath, amsfonts, amssymb}
\usepackage{dsfont, bbm}
\usepackage{stmaryrd}

\usepackage{amsthm}

\usepackage{graphicx}
\usepackage{subcaption}
\usepackage{tikz}
\usepackage{pgfplots}
\usepackage{pgfplotstable}
\pgfplotsset{compat=newest}
\usepgfplotslibrary{patchplots}

\usepackage[inline, shortlabels]{enumitem}

\usepackage{xcolor}
\usepackage[draft]{fixme}

\usepackage[authoryear]{natbib}

\usepackage{authblk}
\usepackage{verbatim}

\newcommand{\cb}{{\mathcal B}}

\newcommand{\cd}{{\mathcal D}}

\newcommand{\ce}{{\mathcal E}}
\newcommand{\cf}{{\mathcal F}}

\newcommand{\cp}{{\mathcal P}}

\newcommand{\crr}{{\mathcal R}}

\newcommand{\E}{{\mathbb E}}
\newcommand{\F}{{\mathbb F}}

\newcommand{\N}{{\mathbb N}}
\renewcommand{\P}{{\mathbb P}}

\newcommand{\R}{{\mathbb R}}

\newcommand{\dom}{{\mathrm{dom}}\;}

\newcommand{\id}{\mathrm{Id}}

\newcommand{\inv}[1]{\mathop{\frac{1}{#1}}\nolimits}

\newcommand{\indep}{\mathrel{\perp \!\!\! \perp}}

\newcommand{\ind}{\mathbf{1}}
\newcommand{\rd}{\mathrm{d}}

\newcommand{\Cb}{C^\mathrm{cb}}
\newcommand{\Cip}{C^\mathrm{ip}}

\newcommand{\nsamples}{50}

\theoremstyle{plain}
\newtheorem{theo}{Theorem}
\newtheorem*{theo*}{Theorem}

\newtheorem{cor}[theo]{Corollary}
\newtheorem*{cor*}{Corollary}
\newtheorem{prop}[theo]{Proposition}
\newtheorem{lem}[theo]{Lemma}

\newtheorem{defi}{Definition}
\theoremstyle{remark}
\newtheorem{rem}{Remark}
\newtheorem{exple}{Example}

\title{Csisz\' ar indices and interpolating copulas}

\author[2]{Cristina Butucea} 
\author[3]{Jean-Fran\c cois Delmas} 
\author[1]{Anne Dutfoy} 
\author[1,2,3]{Antoine Schoonaert} 

\affil[1]{EDF Research \& Development, 91120 Palaiseau, France.}
\affil[2]{CREST, ENSAE, Institut Polytechnique de Paris, 91120 Palaiseau, France.}
\affil[3]{CERMICS, ENPC, Institut Polytechnique de Paris, CNRS, Marne-la-vall\'ee, France.}

\date{}

\begin{document}

\maketitle

\begin{abstract}
We study various properties of $f$-divergences and Csisz\'ar indices between two probability distributions in very general setups for the  convex function $f$ and for the probability distributions. We establish general structural properties of $f$-divergences and show how they are inherited by the associated Csisz\'ar indices, including monotonicity and invariance under suitable transformations. 

We also study the relationship between Csiszár indices and copula representations of random vectors. When the marginal distributions have atoms, the copula representation is not unique and the Csisz\'ar index of the transformed vectors may increase. We build a large family of interpolating copulas which  minimize the Csisz\'ar index and thus preserve the dependence structure of the initial vector.
\end{abstract}

\medskip

\noindent\textbf{Keywords:} Csisz\'ar index \and Copula \and Dependence measure \and $f$-divergence \and Interpolating copulas \and Invariance and monotonicity of Csisz\'ar indices.

\medskip
The research is partially supported by the ANRT fellowship No. 2024/0845.

\section{Introduction}
\label{sec:intro}

Quantifying the discrepancy between probability distributions is a fundamental problem in probability, statistics, and information theory. Among the most general and widely used classes of such measures are the $f$-divergences introduced by \cite{csiszar1964informationstheoretische} and later studied extensively in statistics and information theory, see references monographs such as \cite{cover2006element}, \cite{polyanskiy2025information}.
The $f$-divergences provide a unified framework for many classical divergences, including the Hellinger distance, the Kullback–Leibler, the $\chi^2$, and the Jensen–Shannon divergences. Their convex-analytic structure leads to several important properties that we state here with the particular care of making the setup the most general. 

A natural application of $f$-divergences is the quantification of statistical dependence between two random variables (or vectors) through the Csisz\'ar index. The Csisz\'ar index measures the discrepancy between the product of marginal distributions and the joint distribution of the two random variables (or vectors), and can be seen as the distance to independence of a joint distribution. These indices are often used to characterize or test independence between random vectors. As we show in this work, these indices inherit several structural properties—such as invariance and monotonicity—from the underlying $f$-divergence. In Theorem~\ref{thm:increase-Df}, we show that any measurable transformation on measures decreases the $f$-divergence. Moreover, we provide a necessary and sufficient condition for invariance (equality). 
Such properties need to be well understood in applications were these indices are used as measures of dependence and generalize covariance-based sensitivity indices such as Sobol index.

Moreover, we investigate the connection between  the Csisz\'ar indices and the various copula representations of random vectors. Copulas provide a convenient framework for separating marginal behavior from dependence structure; see for instance the monograph \cite{nelsen2006introduction}. In the case of continuous marginal distributions, it is often stated that the Csisz\'ar index between two random vectors coincides with the $f$-divergence between their joint copula and the product copula. However, this identity may fail when the marginal distributions have atoms, since in that case the copula representation of the joint distribution is no longer unique.

It has already been noticed by~\citet[Theorem~2]{lin2025checkerboard}  that the entropy of the copula of an arbitrary vector is maximized by the checkerboard copula. Their construction uses probability integral transform of the marginal distributions defined by \cite{ruschendorf2013mathematical}  with uniform random variables which are independent for each marginal distribution. 
Building on the checkerboard copula, we introduce in Definition~\ref{defi:interpol-cop} a broader class of interpolating copulas, obtained by probability integral transform with different and possibly dependent uniform random variables for each atom. This construction generalizes the checkerboard approach and provides a flexible family of copulas that interpolate the discrete structure of the marginals while preserving their distributions. 
In Theorem~\ref{thm:interpolating_minimise_fdiv}, we show that the Csiszár index may increase with the choice of copula but it is minimized by these interpolating copulas which preserve the dependence structure. 

The remainder of this paper is organized as follows. 
Section~\ref{sec:general_notation} recalls the general notation used throughout the paper and provides background on convexity and measures. 
Section~\ref{sec:csiszar_divergence} introduces the definitions and properties of $f$-divergences. 
Section~\ref{sec:csiszar_index} is devoted to Csiszár indices, where we establish invariance and monotonicity properties. 
 More technical proofs are deferred to Section~\ref{sec:proof_introduction}.

\section{Notations and definitions}
\label{sec:general_notation}
We  set  $\R^*_+=  (0,  +\infty )$,  $\R^*=\R^*_+  \cup  (-\R^*_+)$  and
$\bar \R=[-\infty , +\infty ]$. For $x\in \R$, we set $x_+=\max(x,0)$
and $x_-=\max(-x, 0)$.

All  the   measures  are   understood  to   be  positive   measures.  Let
$(\Omega, \cf, \mu)$  be a measured space.  Let $f$  a real-valued or an
$\bar  \R$-valued  measurable  function  defined  on  $\Omega$.  We  set
$f_+=  \max(f, 0)$  and  $f_-= -  \min  (f,  0)$.  We  say  that $f$  is
quasi-integrable    if    $\int_\Omega    f_+    \,    \rd    \mu$    or
$\int_\Omega f_- \, \rd \mu$ is finite (notice that $f$ is integrable if
those two integrals are both finite).

Let $(\Omega, \cf)$  and $(E, \ce)$ be a measurable  spaces and $f$
an  $E$-valued   function  defined  on  $\Omega$.    The
function $f$ is measurable if $\sigma(f)\subset \cf$, where
$\sigma(f)=\{  f^{-1}(A)\, \colon\,  A \in  \ce\}$ is  the  $\sigma$-field
generated                 by                $f$, and
bi-measurable if $f(A)\in \ce$ for all $A\in \cf$.

When $\mu$ is  a probability measure, we shall  see measurable functions
defined on $\Omega$ as random variables and thus write $\E_\mu[g]$, when
it is  meaningful, for the  integral of the $\bar  \R$-valued measurable
function $g$ defined  on $\Omega$ w.r.t.\ the  probability measure $\mu$
and $\E_\mu[g\, |\,  f]$ for the expectation  conditionally on the
$\sigma$-field  $\sigma(f)$   (it  is defined  $\P_\mu$-a.s.,
$\sigma(f)$-measurable  and  can  be  written  as  a  function  of
$f$, that is, $\E_\mu[g\, |\, \varphi]=g'\circ f$, with $g'$
a $\bar \R$-valued measurable function defined on $E$).

Let $\id$ denote the identity function and $\ind_A$ the indicator function of the set $A$.

\subsection{On convexity}
\label{sec:on_convexity}
 A  function  $\varphi$ defined  on $\R^d$  is  a proper  convex
function  if  $\varphi(\R^d)  \subset (-\infty,  +\infty]$,  there  exists
$x\in \R^d$ such that $\varphi(x)\in \R$ and $\varphi$ is convex, that is,
for all $ x,y\in \R^d$ and $ \alpha\in [0, 1]$:
\begin{equation}
  \label{eq:def-proper-cvx}
\varphi(\alpha x +  (1-\alpha) y) \leq  \alpha \varphi(x) + (1-\alpha)
\varphi (y).
\end{equation}
Let  $\varphi$ be  a  proper  convex function  defined  on  $\R$. It  is
measurable and continuous on the interior of its domain, where the domain
of a function $\varphi$ is defined by:
\[
  \dom(\varphi)=\varphi^{-1}(\R).
\]
The  limit  of  $\varphi$  at  $+\infty$  (resp.  $-\infty$)  exists  in
$\bar  \R=[-\infty ,  +\infty ]$  and is  denoted by  $\varphi(+\infty)$
(resp. $\varphi(-\infty)$); thus  a proper convex function  on $\R$ will
be  understood to  be  also  defined on  $\bar  \R$  by continuity.  The
function  $\varphi$  is  strictly  convex   at  $z\in  \R$  if  for  all
$x,y\in   \R\setminus\{z\}$   and   $\alpha   \in   (0,1)$   such   that
$z = \alpha x +  (1- \alpha) y$ the inequality~\eqref{eq:def-proper-cvx}
is strict; in this case we also have $z\in \dom(\varphi)$.
Notice we don't assume $\varphi$ to be continuous at finite boundary
point of its domain.

\medskip

For  an   $\bar  \R$  valued   random  variable   $X$,  we  say   it  is
quasi-integrable  if  $\E[X_+]$  or  $\E[X_-]$ is  finite  (notice  that
$X$ is integrable if  and only if both $\E[X_+]$ and $\E[X_-]$  are finite, and
$X$ is finite a.s.). We give the Jensen inequality for quasi-integrable
random variable in  Lemma~\ref{lem:equality_jensen}, Section~\ref{sec:proof_introduction}.

\subsection{On measures}

Let $\nu$ and $\mu$ be  measures on $(\Omega, \cf)$. They
are mutually singular, denoted by  $\nu \perp \mu$,  if there exists $A\in \cf$ such that
$\nu(A)=\mu(A^c)=0$. The measure $\nu$ is  absolutely continuous
w.r.t.\ $\mu$, denoted by $\nu\ll \mu$,
if for all $A\in \cf$ such that $\mu(A)=0$, then we have $\nu(A)=0$. The
measures $\nu$ and $\mu$ are mutually absolutely continuous if $\nu \ll
\mu$ and $\mu \ll \nu$.
Let $M_+$ denote the set of $\sigma$-finite measures on $(\Omega, \cf)$. For $\mu, \nu \in M_+$,  the Radon-Nikodym
theorem asserts that if $\nu \ll \mu$, then there exists a nonnegative
measurable function $h$, denoted by $\rd \nu / \rd \mu$, defined on
$\Omega$ such that $\nu =h \mu$.  Notice the function $h$ is uniquely  defined up
to a measurable set of zero $\mu$-measure.

For   $\mu,  \nu\in   M_+$,   by  considering   the  reference   measure
$\lambda=\mu  +\nu$, we get that  $\mu$  and $\nu$  are absolutely  continuous
w.r.t.\  $\lambda$ and  $\sigma$-finite. Thus, there  exists two  nonnegative
measurable functions  $p$ and $q$  (uniquely defined up to  a measurable
set   of   zero   $\lambda$-measure)  such   that   $\nu=p\lambda$   and
$\mu=q \lambda$. By construction, we have  $p+q=1$ $\lambda$-a.e., and, with the convention
$0/0=\infty /\infty =0$ and setting $h=p/q$ we get:
\begin{equation}
  \label{eq:Leb-decomp}
  \rd \nu = h \, \rd \mu + \ind_{\{h=+\infty \}}\, \rd \nu
  \quad\text{and}\quad
  \rd \mu = \inv{h} \, \rd \nu + \ind_{\{h=0 \}}\, \rd \mu.
\end{equation}
The function $h$, which is   defined $(\mu+\nu)$-a.e.\ and takes values
in $[0, +\infty ]$, is uniquely
characterized by the Lebesgue decomposition~\eqref{eq:Leb-decomp}; we
also have  $\mu(h=+\infty )=0
$ and $\nu(h=0)=0$.
For example, we have $\nu \ll \mu$ if
and only if $\nu(h=+\infty )=0$, and in this case one usually writes $h=
\rd \nu/ \rd \mu$ (with $h$ defined $\mu$-a.e.).
With an  abuse  of
notation, \textbf{even if $\mu, \nu\in  M_+$ are not mutually absolutely
  continuous}, we write:
\begin{equation}
  \label{eq:Leb-dd2}
  \frac{\rd \nu}{\rd \mu}= h
  \quad\text{and}\quad
  \frac{\rd \mu}{\rd \nu}=\inv{  h}
  \quad \text{with $h$ defined $(\mu+\nu)$-a.e..}
\end{equation}

\medskip

Let $(\Omega, \cf)$  and $(E, \ce)$ be a measurable  spaces and $\varphi$
an  $E$-valued   function  defined  on  $\Omega$.    We  say
that an $\bar \R$-valued measurable  function $f$ defined on $\Omega$ is
$\varphi$-measurable   if   $\sigma(f)   \subset   \sigma(\varphi)$   or
equivalently, there  exists an $\bar \R$-valued  measurable function $g$
defined on  the measure space  $(E,\ce)$ such that $f=g  \circ \varphi$.
If $\varphi$ is measurable (that is, $\sigma(\varphi) \subset \cf$), the
push-forward    measure   of    $\mu$   by    $\varphi$,   denoted    by
$\varphi_\sharp  \mu$  or $\mu\circ  \varphi^{-1}$,  is  the measure  on
$(E, \ce)$ defined by  $\varphi_\sharp \mu(A)=\mu (\varphi^{-1}(A))$ for
all $A\in  \ce$. In particular, when $(E, \ce)=(\Omega, \cf)$, we say that $\varphi$ is measure
preserving for $\mu$ if $\varphi_\sharp \mu=\mu$.

For  $g$ a non-negative  measurable function  defined on
$E$, we have:
\begin{equation}
   \label{eq:chgt-var}
  \int_\Omega g\circ \varphi \, \rd \mu= \int_ E g \, \rd \varphi_\sharp
  \mu.
\end{equation}
When $\mu$ is a probability measure, and $X$ a random variable defined on $\Omega$, its distribution $X_\sharp \mu$ is simply denoted by $P_X$.

\medskip

We say that a $[0, +\infty ]$-valued map $K$, also denoted $K(x, \rd
y)$
is a kernel on $\Omega\times E$
if for all $x\in \Omega$, the map $A\mapsto K(x,A)$ defined on $\ce$ is a  measure, and for all $A\in \ce$ the
   function $x\mapsto K(x, A) $ defined on the measurable space
   $(\Omega, \cf)$ is measurable. The kernel $K$ is a Markov or a probability kernel
   if $K(x, E)=1$ for all $x\in \Omega$.

Let $X$ and $Y$ be two random variables defined on the same probability space.
The  conditional distribution of
$Y$ given $X$, when it  exists  (this is for example the case  if $Y$ takes values
in  a  Polish space,  see~\citet[Theorem~1.23]{k:rm}),  denoted by
$P_{Y \vert X}$, is a Markov kernel, and we have:
\begin{equation}
   \label{eq:proba-decomp}
P_{(X,Y)}(\rd x, \rd y)= P_X(\rd x) \, P_{Y\vert X=x}(\rd y).
\end{equation}

\section{The $f$-divergence}

\label{sec:csiszar_divergence}

In this section, we recall the general definition of $f$-divergence and the most currently
used particular cases that simplify this definition. We recall the fundamental properties
and also list results on the invariance and on the contraction properties of the $f$-divergence. We
take particular care to state these results under weakest sets of assumptions.

Let $\F$ be the set of proper convex functions $f$ on $\R$ such that
$\dom(f) \subset \R_+$, $f$ is right continuous at $0$, $f(1)=0$ and $1$ belongs to the
interior of $\dom (f)$.

Notice that $f(0)=\lim_{t \rightarrow 0+}  f(t)$ belongs to $(-\infty, + \infty]$.
For $f\in \F$, we uniquely define its  convex conjugate $f^*\in \F$ by:
\begin{equation}
\label{eq:def_convex_conjugate}
    f^*(t) = t f\left(\inv{t}\right)
\quad\text{for}\, t\in \R_*^+.
\end{equation}
Notice that  $f^*(0)= \lim_{t  \to +\infty}  f(t)/t $,  and that  if the
function    $f$    is   differentiable    on    $(1,    +\infty   )\subset
\dom (f)$,    then
$f^*(0)= \lim_{t  \to \infty}  f'(t)$ (also  denoted by  $ f'(\infty)$).
Considering  proper  convex  functions  is motivated  by  convex  duality
approaches  for the study of $f$-divergences, see~\cite{agrawal2021optimal}  in  this
direction.

\begin{rem}
    The conjugate function $f^*$ of $f$ shall not be confused with  the
    Fenchel conjugate of $f$ which is defined on $\R$ by $t\mapsto  \sup_{x \in \R } tx - f(x)$.
\end{rem}
Let  $\cp=\cp(\Omega, \cf)$  be the  set  of probability  measures on  a
measurable space $(\Omega, \cf)$. We shall assume that
 $\cf$ is  not
      reduced  to the  trivial  $\sigma$-field $\{\emptyset,  \Omega\}$,
      that is, $\cp$ is not a singleton.

\begin{defi}
Let  $P, Q\in \cp$ be two probability
measures.
Using the convention given in~\eqref{eq:Leb-dd2} where $\rd P/\rd Q$ is
defined $(P+Q)$-a.e.,  the $f$-divergence,
for $f\in \F$,  between
$P$ and $Q$ is defined by:
\begin{equation}
\label{eq:definition_df0}
 D_f(P \Vert Q) = \int_\Omega f\left( \frac{\rd P}{\rd Q} \right) \rd
 Q + f^*(0) \, P \left( \frac{\rd P}{\rd Q} =+\infty \right),
 \end{equation}
 with the convention that the last term $f^*(0) \, P \left( {\rd P}/{\rd Q}
   =+\infty \right)$ is zero   if $P\ll Q$. In particular, we have the usual definition if $Q$ dominates $P$:
 \[
   P\ll Q \implies D_f(P \Vert Q) = \int_\Omega f\left( \frac{\rd P}{\rd Q} \right) \rd
   Q .
 \]
\end{defi}
We can rewrite the  $f$-divergence as follows using  a $\sigma$-finite
measure $\lambda$ on $(\Omega, \cf)$ which dominates  $P$ and $Q$ (that
is, $P+Q \ll \lambda$). Denoting  by $p$ and $q$ the corresponding
densities of $P$ and  $Q$ w.r.t.\ $\lambda$, we get:
\begin{equation}
\label{eq:definition_df}
D_f(P\Vert Q) = \int_{q>0} f\left(\frac{p}{q}\right)\, q \, \rd \lambda+f^*(0)\int_{q=0}  p \, \rd \lambda ,
\end{equation}
with the convention that $f^*(0)\int_{q=0}  p \, \rd \lambda=0$ \ if $P(q=0)=0$.
Note that the quantity $D_f(P\Vert Q)$ is well defined by
Remark~\ref{rem:jensen} on Jensen inequality as ${\rd P}/{\rd Q}$ belongs to $L^1(Q)$ and
that it belongs to $(-\infty , +\infty ]$.

Using~\eqref{eq:def_convex_conjugate}
and~\eqref{eq:definition_df}, we can also rewrite  $D_f(P\Vert Q)$
in a symmetric way:
\begin{equation}
   \label{eq:decom_Df}
  D_f(P\Vert Q)
= \int_{0\leq p<q} f\left(\frac{p}{q}\right) \, q  \, \rd \lambda
     + \int_{0\leq  q<p} f^*\left(\frac{q}{p}\right)\, p \, \rd \lambda.
\end{equation}

\begin{rem}[Domain of $f\in \F$]
  \label{rem:domain}
Let a  proper convex function $f$ on $\R$ with domain in $\R_+$ be such that
  $f(1)=0$ and 1 does not belong to the interior of $\dom(f)$. In such a case
  $\dom(f)$ is either a subset of $[1, +\infty )$ or of $ [0,
  1]$, with $f^*$ still defined by~\eqref{eq:def_convex_conjugate}. We
  deduce from~\eqref{eq:decom_Df}  that either the first term or the second term
   on the right-hand side is infinite, meaning that $  D_f(P\Vert Q)=+\infty $, as soon as
  $P\neq Q$. Thus, the assumption we made at the beginning of the section that $1$ belongs to the
  interior of $\dom(f)$ for  $f\in \F$ allows to avoid such trivial cases.
\end{rem}

\begin{rem}[$f\in \F $ affine]
  \label{rem:affine}
We recall that $f\in \F$ is affine if  $f=c(\id -1)$ on $\R_+$ for some
$c\in \R$.
  For $f\in \F$, it is easy to check that the quantity $(f+f^*)(0)$ belongs to
$[0, +\infty ]$ and  that:
\[
  f\quad\text{is affine}\quad
  \Longleftrightarrow\quad
  (f+f^*)(0)=0.
\]
Indeed, $(f+f^*)(0)= \lim_{\varepsilon \rightarrow 0+} (1+\varepsilon)^{-1}
 f(\varepsilon) + \varepsilon (1+\varepsilon)^{-1} f(1/\varepsilon) \geq
 f(1)=0$ by the convexity of $f$, and the equality holds if and only if
 $f$ is affine on $[\varepsilon, 1/\varepsilon]$ for any $\varepsilon>0$.

Furthermore if $f\in \F$ is affine,  we get
that  $f^*(t) = t \cdot c (\frac 1t - 1) =-f(t)$ on $\R^+$. We deduce
from~\eqref{eq:decom_Df} with $\lambda=P+Q$
that $   D_f(P\Vert Q) = - c \int (p-q)_- \rd \lambda+ c \int (p-q)_+ \rd \lambda=0$ for all $P, Q\in \cp$.
Indeed, $\int(p-q) \rd \lambda = 0$. In conclusion, the $f$-divergence is invariant if we add
to $f$ an arbitrary affine function.
\end{rem}

\medskip

\begin{rem}{\bf (On the second term in the right-hand side sum
  of~\eqref{eq:definition_df0} and~\eqref{eq:definition_df}) } 
Consider  two probability measures, $P$ and
$Q$, on a measured space $(\Omega,\cf,\lambda)$ with respective  densities  $p$ and $q$ w.r.t.\ the reference measure $\lambda$.
When $P$ is not absolutely continuous w.r.t.\ $Q$, the
addition of the second term on the right-hand side term in~\eqref{eq:definition_df}
is necessary to ensure the $f$-divergence is indeed
nonnegative, see Lemma~\ref{lem:f_div_properties}~\ref{enum:f_div_prop_nonnegativity}
below. This is also motivated by the following argument,
see~\cite{csiszar1964informationstheoretische}, where
for $x\in \{p>0, q=0\}$ one intuitively has:
\[
q(x) \, f\left( \frac{p(x)}{q(x)}\right) = \lim_{\varepsilon \to 0^+}
\varepsilon \, f\left( \frac{p(x)}{\varepsilon}\right) = p(x) \lim_{u
  \to +\infty} \frac{f(u)}{u} = p(x) \, f^*(0).
\]
\end{rem}

According to Section~2.2 in \cite{rahman2016f}, we have the following
Lemma. Recall that the $\sigma-$algebra  $\cf$ is  not
      reduced  to the  trivial  $\sigma$-field $\{\emptyset,  \Omega\}$.
\begin{lem}[Properties of the $f$-divergence]
  \label{lem:f_div_properties}

  Let $f\in \F$ (and thus $f(1)=0$) and $P,Q\in \cp$. We have:
\begin{enumerate}[(i)]
    \item \label{enum:f_div_prop_nonnegativity} Nonnegativity: $D_f(P \Vert Q) \geq  0$ and $D_f(P\Vert P)=0$.
    \item \label{enum:f_div_prop_reflexivity} Positivity:  If $f$ is
      strictly convex at $1$, then:
      \[
        D_f(P \Vert Q) = 0
        \quad \Longleftrightarrow\quad
        P = Q.
        \]
    \item \label{enum:div_prop_supremum} Supremum: $D_f(P \Vert Q) \leq
      (f+f^*)(0) $, and  $D_f(P \Vert Q) = (f+f^*)(0)$ if $P$ and $Q$ are mutually singular measures, $P \perp Q$.
    \item   \label{enum:div_prop_condition_sup}    Condition   for   the
      supremum:   If  $(f+f^*)(0)$   is   positive   and  finite,
      then:
      \[
        D_f(P \Vert Q) = (f+f^*)(0)
        \quad \Longleftrightarrow\quad
        P \perp Q.
        \]
    \item \label{enum:f_div_prop_duality} Duality: $D_f(P \Vert Q) = D_{f^*}(Q\Vert P)$.
    \item \label{enum:f_div_prop_invariance} Invariance up to affine functions: Let   $g = f +
      c(\id-1)$ with $c\in \R$, then 
      $$
      g\in \F \quad \text{and}  \quad D_{g}(P\Vert Q) = D_{f}(P\Vert Q).
      $$
    \item  \label{enum:f_div_prop_symmetry} Symmetry:   $D_f(P  \Vert Q) = D_{f}(Q\Vert  P)$ for all $P,Q\in  \cp$ if
      and  only if  there exists  $h\in  \F$ and  $c \in  \R$ such  that
      $f=h+h^* + c(\id-1)$. Then, we also have $\dom(f)=\dom(f^*)$.

    \item \label{enum:f_div_prop_range} Range:
  If  $\R_+^*\subset \dom(f)$, then
  $\{D_f(P\Vert Q) \, \colon\, P,Q\in \cp\}= [0, (f+f^*)(0)]
  $.
\end{enumerate}
\end{lem}

For   the   readers'   convenience,   we  provide   a   short   proof   in
Section~\ref{sec:proof_introduction}.

\begin{rem}[Restriction to dominated probability measures]
  \label{rem:ext-to-Pl}
  The                          properties {\it ~\ref{enum:f_div_prop_symmetry}}
  and{\it ~\ref{enum:f_div_prop_range}}  hold  also  with  $\cp$  replaced  by
  $\cp_\lambda\subset\cp$ the  subset of probability measures  which are
  dominated  by a  non-trivial  $\sigma$-finite  measure $\lambda$.  See
  Remark~\ref{rem:PL-ext} below for a proof.
\end{rem}

\begin{rem}[Examples of $f$-divergence]
   \label{rem:exple-Df}
   We  provide  in Table~\ref{tab:list_of_fdivergence}  some  well-known
   $f$-divergences,  see~\cite{ca:abc-divergence},  and we  simply  write
   $D_\star=D_{f_\star}$ for the name $\star$ in this table.

\begin{table}[htbp]
    \renewcommand{\arraystretch}{1.25}
    \centering
    \begin{tabular}{|c|c|}
        \hline
        $f$-divergence & Expression of $f$
         \\
        \hline
        Kullback-Leibler divergence & $f_\mathrm{KL}(t) = t \log(t)$
        \\
        Conjugate Kullback-Leibler divergence & $f_{\mathrm{KL}^*}(t) = -\log(t)$
        \\
      Total variation distance & $f_\mathrm{TV}(t) = \vert t-1\vert$ \\
        Squared Hellinger distance & $f_\mathrm{H}(t) = \left(\sqrt{t}-1 \right)^2$\\
        Pearson $\chi^2$ divergence & $f_\mathrm{P}(t)= t^2-1$ \\
        Neyman $\chi^2$ divergence & $f_\mathrm{N}(t)= \frac{1-t^2}{t}$ \\
     $\alpha$-divergence, 
      $\alpha\in \R \setminus \{0,1\}$
                       & $f_\alpha(t) = \frac{t^\alpha -\alpha t
                         -(1-\alpha)}{\alpha (\alpha -1)}$     \\
      Le Cam distance & $f_\mathrm{LC}(t)= \frac{(1-t)^2}{2t+2}$ \\
        Jensen-Shannon divergence & $f_\mathrm{JS}(t)= t \log\left(\frac{2t}{t+1}\right) + \log\left(\frac{2}{t+1}\right)$ \\
        \hline
    \end{tabular}
    \caption{Some examples of $f$-divergences.}
    \label{tab:list_of_fdivergence}
\end{table}

We  have the  conjugate  relation for  the (Conjugate)  Kullback-Leibler
divergence  $f_\mathrm{KL}^*= f_{\mathrm{KL}^*}$, as well as  for  the (Neyman  and
Pearson)  $\chi^2$ divergence  $f_\mathrm{P}^*=  f_\mathrm{N}$; moreover, for  the
$\alpha$-divergence we have $f_\alpha^*=f_{1-\alpha}$.    The  Total  variation
distance, the  Squared Hellinger distance,  the Le Cam distance  and the
Jensen-Shannon  divergence  are  symmetric  as  $f_\star^*=f_\star$  for
$\star\in  \{\mathrm{TV},  \mathrm{H}, \mathrm{LC},  \mathrm{JS}\}$  and
furthermore the Total variation distance  and the square root divergence
$(D_{\star})^{1/2}$                                                    for
$\star\in   \{  \mathrm{H},   \mathrm{LC},  \mathrm{JS}\}$   are  indeed
distances       on       $\cp$.

According       to~\citet[Theorem~1
and~2]{kafka1991powers},  if $f\in  \F$ is  strictly convex  on $\R_+^*$
(thus   $\R_+^*  \subset   \dom(f)$)   with  $f=f^*$   (thus  $f>0$   on
$\R_+\setminus\{1\}$)             and              such             that
$t\mapsto (1-  t^\beta)^{1/\beta}/ f(t)$  is non-increasing on  $[0, 1)$
for  some  $\beta\in  \R_+^*$   (thus  $f(0)=f^*(0)$  is  finite),  then
$(D_f)^{\beta}$  is a  distance  on $\cp$.   Notice  those conditions  are
satisfied for the $f_\star$ with $\star\in \{ \mathrm{H}, \mathrm{LC}\}$
and  $\beta=1/2$. See  also~\cite{ov:2003}  for further  results in  this
direction, including in particular the fact that $(D_{\mathrm{JS}})^{1/2}$
is   indeed  a   distance  (this   appears   as  the   limiting  case   of the function
$t\mapsto \gamma(  \gamma -1)^{-1} \left(  (1+t^\gamma)^{1/\gamma}  - 2^{(1-\gamma)/\gamma}
(1+t)\right)$ as $\gamma\in \R_+^*\setminus\{1\}$ goes to $1$).
\end{rem}

\begin{rem}[$f_\alpha$-divergence and Rényi divergence]
   \label{rem:renyi-div}
   For the $\alpha$-divergence, one can define $f_0$ by continuity at
   $0+$ as $f_0(t)=-\log (t) + (t-1)$ and $f_1$ by continuity at $\alpha=1$ or
   conjugation   $f_1(t)=f_0^*(t)=t\log(t)    +   (1-t)$.    Thanks   to
   Lemma~\ref{lem:f_div_properties}{\it ~\ref{enum:f_div_prop_invariance}} the
   $1$-divergence  corresponds   to  the   Kullback-Leibler  divergence,
   $D_1=D_\mathrm{KL}$   and  the   $0$-divergence   to  the   Conjugate
   Kullback-Leibler divergence.

Let us mention that  the
$f$-divergence associated to $f_\alpha$ and that associated to the function $(t^\alpha -1)/\alpha (\alpha-1)$
coincide thanks to this invariance property. Indeed, $f_\alpha$ and the previous function only differ by an
affine function and they result in the same $f$-divergence. Therefore the  Rényi's divergence
\cite{renyi1961measures} of order
 $\alpha \in (0,+\infty )\setminus \{1\}$ defined by:
 \[
R_\alpha(P\Vert Q) =
\inv{\alpha-1} \log \left( \int_\Omega p^\alpha q^{1-\alpha} \, \rd
  \lambda \right)
 \]
can be written as:
\begin{equation*}
R_\alpha(P\Vert Q)= \inv{\alpha-1} \log \Big( 1+ \alpha(\alpha-1)
  D_{\alpha} (P\Vert Q)\Big),
\end{equation*}
where $p,q$ represent the density of $P, Q\in \cp$ w.r.t.\ a dominating
measure $\lambda$.
By  letting   $\alpha$  go  to  $1$,  we   get  by   continuity  that
$R_1 =D_1 = D_\mathrm{KL}$.  For $\alpha=1/2$, we obtain the  Hellinger
divergence  $D_{1/2}= 2 D_H$,  and $R_{1/2}  $ is  two times  the so-called
Bhattacharyya divergence, see \cite{kailath2003divergence}, which is not a
distance   as  it  does  not  respect  the  triangular
inequality  (but  recall  instead  that $D_H^{1/2}$  is  a  distance).
\end{rem}

We give in  the  following proposition a sufficient condition for the
invariance of the $f$-divergence between the push-forward measures via a measurable function,
whose      proof     is      given  in Section~\ref{sec:proof_introduction}.

\begin{prop}[Invariance of the $f$-divergence]
  \label{prop:inv-gen}
    Let $P,Q$ be two probability measures on $(\Omega, \cf)$. Let $(E, \ce)$
    be a measurable space and $\varphi: \Omega \rightarrow E$ a measurable
    function. If $\rd P/\rd Q$ (in the
    sense of~\eqref{eq:Leb-dd2}), or equivalently $\rd Q/ \rd P$, is
    $\sigma(\varphi)$-measurable then we have  $\rd P/
    \rd Q= (\rd  \varphi _\sharp P/ \rd \varphi_\sharp Q)\circ \varphi$,
    and  for $f\in \F$:
    \[
      D_f(\varphi_\sharp P \, \Vert \, \varphi_ \sharp Q) =  D_f(P \Vert Q).
    \]
\end{prop}

We  directly deduce  the following result.

\begin{cor}[Invariance of the $f$-divergence for random variables]
  \label{cor:inv-gen-XY}
  Let $X,Y$ be  two random variables taking values in $(\Omega, \cf)$
     with respective distribution $P_X$ and $P_Y$. Let $(E, \ce)$ be a
     measurable space and $\varphi$ be a $E$-valued measurable function
     defined on $\Omega$. If $\rd P_X / \rd P_Y$ can be written as a
     measurable function of $\varphi$, then we have for all $f\in \F$:
     \[
        D_f( P_{\varphi(X)} \, \Vert \,  P_{\varphi(Y)})= D_f(P_X \Vert P_Y ).
    \]
\end{cor}

Let us consider now  the  case   where  $\rd   P  /   \rd  Q$   is  not
$\sigma(\varphi)$-measurable.  The  next  result implies  in  particular
\citet[Theorem 7.4]{polyanskiy2025information} and does not rely on the
existence of  conditional distribution  but the  proof follows  the same
lines.  The  identification of the  ratio densities of the  push-forward
probability measures and the if and only if condition are part of the folklore; the proof
is postponed to Section~\ref{sec:proof_introduction}.

\begin{theo}[Push-forward measures and  $f$-divergence]
  \label{thm:increase-Df}
  Let $P,Q$ be  two probability measures  on a  measurable space
  $(\Omega,  \cf)$ and  $\varphi$ be  a measurable  function defined  on
  $\Omega$ (taking values in a  general measurable space $(E, \ce)$). We
  have for $f\in \F$ that:
\begin{equation}
   \label{eq:D-phi}
   D_f(\varphi_\sharp P \, \Vert \, \varphi_ \sharp Q) \leq
   D_f(P    \Vert Q)
    \end{equation}
    and, with $\lambda=(P+Q)/2$, we have $\lambda$-a.s.:
 \begin{equation}
   \label{eq:contrib-pq>0}
   \frac{\rd \varphi_\sharp P}{\rd \varphi _\sharp Q}=
   \frac{\E_Q\left[\frac{\rd P}{\rd Q} \, \big |\, \varphi\right]}
   {P\left(\frac{\rd P}{\rd Q}<+\infty  \, \big |\, \varphi\right)}
   \in (0, +\infty )
   \quad\text{on}\quad
        \left\{\E_\lambda\left[\frac{\rd P}{\rd \lambda}\, \frac{\rd
              Q}{\rd \lambda}\, \big |\, \varphi \right]>0
      \right\}.
    \end{equation}

    Furthermore,   if   $f$   is   strictly   convex   on   $\R_+^*$   and
    $ D_f(P \Vert Q)$ is finite, then inequality~\eqref{eq:D-phi} is
    an equality if and only if $\rd P/\rd Q$ is (equal $\lambda$-a.s.\
    to) a measurable function of
    $\varphi$, that  is, there exists a  nonnegative measurable function
    $h$ defined on $E$ such that:
    \[
    \text{$\lambda$-a.s.}\quad \quad  \frac{\rd P}{\rd Q}= h\circ \varphi.
    \]
\end{theo}

The     formula~\eqref{eq:contrib-pq>0}     is  sufficient   to     give
$ {\rd \varphi_\sharp P}/{\rd \varphi _\sharp  Q}$ when $P \ll Q$. Other
cases   might    be   more    delicate:   in   particular    the   event
in~\eqref{eq:contrib-pq>0}  has   zero  measure  if  $P\perp   Q$,  see
Remark~\ref{rem:inv-gen} below.

\begin{cor}[Push-forward measures when $P \ll Q$]
  \label{cor:increase-Df}
  Let $P,Q$ be  two probability measures  on a  measurable space
  $(\Omega,  \cf)$ such that $P \ll Q$  and  $\varphi$ be  a measurable  function defined  on
  $\Omega$ (taking values in a  general measurable space $(E, \ce)$).
 We have for $f\in \F$ that:
\begin{equation}
   \label{eq:D-phi2}
   D_f(\varphi_\sharp P \, \Vert \, \varphi_ \sharp Q) \leq
   D_f(P    \Vert Q)
\quad\text{and}\quad
  \frac{\rd\varphi_\sharp P}{\rd \varphi_\sharp Q}\circ \varphi =
  \E_Q\left[\frac{\rd
         P}{\rd Q}\,|\, \varphi\right] \quad\text{$Q$-a.s..}
\end{equation}

    Furthermore,   if   $f$   is   strictly   convex   on   $\R_+^*$   and
    $ D_f(P \Vert Q)$ is finite, then the inequality in~\eqref{eq:D-phi2} is
    an equality if and only if  $Q$-a.s.\ $ \E_Q\left[(\rd
         P/ \rd Q) \,|\, \varphi\right]=\rd
         P/ \rd Q$.
\end{cor}

\begin{rem}[On related work]
   \label{rem:related-work}
  \begin{enumerate}[(i)]
  \item \label{item:div-kernel} Let $  (U,V)$ and  $ (U',V')$  be random  variables defined  on
    $(\Omega,   \cf)$.   Assume   that  the   conditional  distributions
    $P_{V|U}$         and         $P_{V'|U'}$         exists         and
    \begin{equation}
        \label{eq:kernel-eqal}
        P_{V|U=u}(\rd     v)=P_{V'|U'=u}(\rd      v)
        \quad (P_U+P_{U'})(\rd u)-\text{a.e.}.
    \end{equation}
         By     choosing
    $\varphi(u,v)  = u$,  we deduce  from Proposition~\ref{prop:inv-gen}
    the result in     \citet[Proposition~7.2]{polyanskiy2025information}:
    $$
    D_f(P_{(U,V)}\Vert  P_{(U',V')}) =  D_f(P_U \Vert  P_{U'}).
    $$
  \item     \label{it:phi-inj}      Let
   $(\Omega,  \cf)$ and $(E, \ce)$  be Polish spaces
    and $\varphi:\Omega \mapsto E$  be an injective  function. According
    to \cite{purves1966bimeasurable}, the function $\varphi$ is
    bi-measurable.      We      obtain     $\sigma(\varphi)=\cf$      and
    \citet[Corollary~2.18]{polyanskiy2025information} is then  a direct consequence
    of      Proposition~\ref{prop:inv-gen}   (with  $f=f_\mathrm{KL}$).
  \end{enumerate}
\end{rem}

\begin{rem}[Examples of invariance of the $f$-divergence]
  \label{rem:inv-gen}
  $ $
  \begin{enumerate}[(i)]
  \item
    \label{item:inv-gen-sf=f}
    Let  $X,Y$ be two
     random variables on $(\Omega, \cf)$ and  $\varphi$  a measurable
     function  defined on $\Omega$ and taking values in a measurable
     space $(E, \ce)$.
If  $\sigma(\varphi)=\cf$, then we deduce that for $f\in \F$:
  \begin{equation}
   \label{eq:DfPf=bij}
          D_f( P_{\varphi(X)} \, \Vert \,  P_{\varphi(Y)}) = D_f(P_X \Vert P_Y ).
     \end{equation}
     Notice  that  if  $\varphi$  is a  bi-measurable  injection  (those
     conditions   are  satisfied   when  $\varphi$   is  bijective   and
     $(\Omega,        \cf)$  and $(E, \ce)$ are       Polish        spaces,
     see~\cite{purves1966bimeasurable}),       then        we       have
     $\sigma(\varphi)=\cf$, and thus~\eqref{eq:DfPf=bij} holds.

      \item Let $X,Y$ be two  real-valued
     random variables. We deduce from Theorem~\ref{thm:increase-Df}
     that for all $f\in \F$:
     \[
          D_f( P_{|X|} \, \Vert \,  P_{|Y|}) \leq  D_f(P_X \Vert P_Y ).
        \]
        If furthermore $X$ and $Y$  are symmetric (a random variable $Z$
        is  symmetric  if  $Z$  and   $-Z$  have  the  same  probability
        distribution), then
        $\rd P_X/ \rd P_Y$  is $\sigma(|\cdot|)$-measurable, and we
        deduce from Proposition~\ref{prop:inv-gen}  that
        for all $f\in \F$:
 \[
         D_f( P_{|X|} \, \Vert \,  P_{|Y|})=D_f(P_X \Vert P_Y ).
      \]
      In this  example the absolute  value function $|\cdot|$ is  not an
      injection  and  the  distributions  $P_X$ and  $P_{|X|}$  are  not
      mutually absolutely continuous.

    \item Let $X$ be a positive random variable,  $Y=-X$. We have $P_X
      \perp P_Y$ and $P_{|X|}=P_{|Y|}$ so that
  for all $f\in \F$ which is not affine whe have the strict inequality:
\[
 0=        D_f( P_{|X|} \, \Vert \,  P_{|Y|})<D_f(P_X \Vert P_Y )=(f+f^*)(0).
      \]
      In           this           example,           the           event
      $  \left\{\E_\lambda\left[\frac{\rd P_X}{\rd  \lambda}\, \frac{\rd
            P_Y}{\rd       \lambda}\right]>0       \right\}$,       with
      $\lambda=(P_X+  P_Y)/2$, which  appears in~\eqref{eq:contrib-pq>0}
      is  of  zero  $\lambda$-measure;  notice  also  that the
      denominator       in       this       formula       is       zero:
      $P_X\left(\frac{\rd     P_X}{\rd    P_Y}<+\infty     \,    \big     |\,
        \varphi\right)=0$.

  \end{enumerate}
\end{rem}

We now consider the case studied in~\citet[Theorem~7.4]{polyanskiy2025information}
related to the Markov property; notice however that we do not assume
that the conditional distributions $P_{Y| X=x}(\rd y)$ and $P_{Y'| X'=x}(\rd y)$ exist nor that they are equal.

\begin{prop}[Markov property and $f$-divergence]
\label{prop:markov}
Let two random vectors $(X,Y,Z)$ and $(X',Y',Z')$ take values in
the same product measurable space such that $X$ and $Z$
(resp.\  $X'$ and  $Z'$) are  independent conditionally  on $Y$  (resp.\
$Y'$) and  the conditional probability distributions  $P_{Z|Y=y}(\rd z)$
and      $P_{Z'|Y'=y}(\rd     z')$      exist     and      are     equal
$(P_Y+P_{Y'})(\rd y)$-a.e.. Then, we have:
\begin{equation}
   \label{eq:densite-Markov}
  \frac{\rd P_{(X', Y', Z')}}{\rd P_{(X,Y,Z)}}(x,y,z)=
      \frac{\rd P_{(X', Y')}}{\rd P_{(X,Y)}}(x,y),
          \,
\end{equation}
$ \left(P_{(X', Y', Z')}+ P_{(X,Y,Z)}\right)(\rd x, \rd y,
              \rd z)\text{-a.e.}$. In particular, for all $f\in \F$:
$$
D_f\left(P_{(X', Y', Z')} \,\Vert \,  P_{(X,Y,Z)}\right)=
  D_f\left(P_{(X', Y')} \, \Vert \,  P_{(X,Y)}\right)
  $$ as well as:
\begin{equation}
   \label{eq:ineg-Df-Markov}
D_f\left(P_{(X', Z')} \,\Vert \,  P_{(X,Z)}\right)\leq
  D_f\left(P_{(X', Y')} \, \Vert \,  P_{(X,Y)}\right).
\end{equation}
\end{prop}


\section{The Csisz\'ar index}
\label{sec:csiszar_index}

We discuss in this section the Csisz\'ar index defined as the $f$-divergence between the product and the joint distribution of two random variables. First, we give the consequences of our previous results on $f$-divergences and state invariance and contraction properties of the Csisz\'ar index. Next, we introduce a family of interpolating copulas for distributions which may have atoms. Finally, we show that these interpolating copulas minimise the $f$-divergence between the product of marginal copulas and the joint copula of the initial vector.

\subsection{Introduction}
Recall that for a random variable $Z$, $P_Z$ denote its  probability distribution. 
Let  $(X,Y)$ be a couple of random variables, with $X$ (resp. $Y$) taking values in a measurable space $(\Omega_X, \cf_X)$ (resp. $(\Omega_Y, \cf_Y)$).
We  denote  by $P_X \otimes P_Y$ the product distribution (which is the probability distribution of $(X',Y')$ with $X'$ and $Y'$ independent and distributed as $X$ and $Y$). For a function $f \in \F$, we define the Csisz\'ar index between $X$ and $Y$ by:
\begin{equation}
\label{eq:def_S_f}
    S_f(X,Y) = D_f(P_X \otimes P_Y \Vert P_{(X,Y)}).
\end{equation}
This  quantity   compares  the  joint  distribution and the  product
distribution, and  is a way  to measure  the dependence between  $X$ and
$Y$.  As  with  $f$-divergence,  this quantity  can  be  infinite.  This
Csisz\'ar index has the following  properties inherited from the ones of
the $f$-divergence given in Lemma~\ref{lem:f_div_properties}.
\begin{cor}[Csisz\'ar index properties]
\label{cor:Csiszar_index_properties}
    Let $f \in \F$ (and thus $f(1)=0$), and $(X,Y)$ a couple of random variables.
    \begin{enumerate}[(i)]
    \item \label{enum:S_index_prop_symmetry} Symmetry: $S_f(X,Y) = S_f(Y,X)$.
    \item \label{enum:S_index_prop_nonnegativity} Nonnegativity: $S_f(X,Y) \geq 0$. If $X \indep Y$ (that is, $X$ and $Y$ are independent), then $S_f(X,Y) = 0$.
    \item \label{enum:S_index_prop_independent} Independence: If $f$ is strictly convex at $1$, then $S_f(X,Y) = 0 \implies X \indep Y$.
\end{enumerate}
\end{cor}

\begin{rem}{\bf (The Csisz\'ar index when $P_X \otimes P_Y \ll P_{(X,Y)}$)} 
We assume   that $P_{(X,Y)}$ is absolutely continuous w.r.t.\ a measure $\lambda=\lambda_X
        \otimes \lambda_Y$ on $(\Omega_X \times \Omega_Y, \cf_X\otimes
        \cf_Y)$, and  denote by $p_{(X,Y)}$ the corresponding
        density. We also denote by $p_X$  the corresponding
        marginal probability densities of $X$  w.r.t.\
        $\lambda_X$, and similarly for $Y$.
It is worth noticing that if furthermore 
 $P_X \otimes P_Y \ll P_{(X,Y)}$, 
 then the term  with $f^*(0)$ in~\eqref{eq:definition_df} which appears also in 
  Equation~\eqref{eq:def_S_f} has no contribution, and thus: 
 \begin{equation}
        S_f(X,Y) = \E \left[ f\left(\frac{p_X(X)\, p_Y(Y)}{p_{(X,Y)}(X,Y)} \right) \right].
      \end{equation}
\end{rem}

\begin{rem}{\bf (Csisz\'ar index in terms of conditional distributions)}
\label{rem:csiszar_conditional}
Assume that the conditional distribution $P_{Y \vert X}$ of $Y$ given $X$ exists (which holds, for instance, when $(\Omega_Y, \cf_Y)$ is a Polish space). We define the following nonnegative function $\cd_f(x) = D_f(P_Y \Vert P_{Y|X=x})$, on the measurable space $(\Omega_X, \cf_X)$.
For simplicity, we write by convention $
  D_f(P_Y \Vert P_{Y|X})=\cd_f(X)$. According to~\citet[Proposition~7.2]{polyanskiy2025information}, the function $D_f(P_Y \Vert P_{Y|X})$  is a nonnegative random variable defined on $\Omega_X$ and the Csisz\'ar index admits the following representation in terms of the conditional distribution:
\begin{equation}
  \label{eq:Df-kernel}
  S_f(X,Y) = \E_X \left[ D_f(P_Y \Vert P_{Y|X}) \right],
\end{equation}
where we write $\E_X$ to stress that in the expectation, the randomness comes form the random variable $X$. See also~\cite{rahman2016f} for further results in this direction. 
\end{rem}

\begin{rem}[The Csisz\'ar index when $f = f_{\mathrm{KL}^*}$] 
\label{rem:mutual_information} 
Let $X\in \R^d$ be a real random vector whose probability distribution has density $p_X$ w.r.t.  the Lebesgue measure $\lambda$. When $\E[\log(p_X(X))]$ exists, its entropy, see~\cite{cover2006element}, which can be seen as 
the $f$-divergence between the probability distribution of $X$ and the Lebesgue measure on $\R^d$, is defined by:
    \[
        \label{eq:entropy_definition}
        H(X) = - \int_{\R^d} \log (p_X) \, p_X \, \rd \lambda .
      \]
 
      Let $(X,Y) \in \R^{d_X+d_Y}$ be  a random vector whose probability
      distribution   has  density   $p_{(X,Y)}$   w.r.t.  the   Lebesgue
      measure. Let $p_{Y|X=x}(y)$ denote  the density of the conditional
      distribution   of   $Y$   given   $X=x$   w.r.t.\   the   Lebesgue
      measure. Provided  all the terms are well defined, the  conditional entropy of
      $Y$ given $X$ is defined by:
      \[
          H(Y\vert  X) = - \int_{\R^{d_X+d_Y}}  \log (p_{Y|X=x}(y))  \, p_{(X,Y)}(x,y) \, \rd x \rd y = H(X,Y) - H(X).
      \]
  
  For  $f=f_{\mathrm{KL}^*}$  the conjugate Kullback Leibler, the
  Csisz\'ar index between $X$ and $Y$  coincides with the mutual
  information between $X$ and $Y$, that is, provided all the terms are well defined:
  \[
  S_{ f_{\mathrm{KL}^*}}(X,Y)= H(X) + H(Y)-H(X,Y)=
  H(Y) - H(Y\vert  X). 
  \]
\end{rem}

\subsection{Invariance of the Csisz\'ar index}

We present  in this  section different ways  of transforming  the random
variables $X$ and $Y$ which do not modify the Csisz\'ar index. 

The following corollary  is a direct consequence of
Theorem~\ref{thm:increase-Df}. 

\begin{cor}[Variable transformation reduces the Csisz\'ar index]
    \label{cor:increase_Sf}
  For $i=1, 2$, let  $X_i $  be a random variable taking values in a
  measurable space $\Omega_i $ and
$\varphi_i$ a measurable function defined on $\Omega_i$ and taking
values in a general measurable space. Then, for all $f \in \F$, we have:
\begin{equation}
   \label{eq:Sf-ineq}
   S_f(\varphi_1(X_1), \varphi_2(X_2)) \leq   S_f(X_1,X_2) .
\end{equation}    
 \end{cor}

 We are now interested on conditions on the transformation for which
 the inequality~\eqref{eq:Sf-ineq} is an equality. 
The next result is a generalization of
\citet[Proposition~3.6]{rahman2016f}, were stronger regularity conditions where
assumed on the transformation.

\begin{prop}[Invariance of the Csisz\'ar index]
\label{prop:invariance_csiszar_index}
Let $f \in \F$ (and thus $f(1)=0$). For $i=1,2$, let $X_i$ be a random variable taking values in  a measurable
space $(\Omega_i, \cf_i)$ and $\varphi_i$ a measurable function defined
on $(\Omega_i, \cf_i)$ and taking values in a  mea\-su\-rable
space $(E      _i, \ce_i)$. 
    \begin{enumerate}[(i)]
    \item \textbf{Marginal invariance.} \label{enum:marginal_invariance}
   If $\rd P_{X_1} \otimes \, \rd P_{X_2} / \rd P _{(X_1,X_2)}$
    is $\sigma(\varphi_1) \otimes \sigma(\varphi_2)$-mea\-su\-rable, then
    we have:
    \[
      S_f(X_1, X_2)=S_f(\varphi_1(X_1), \varphi_2 (X_2)). 
    \]

  \item\textbf{Invariance by injection.}
  \label{enum:bijection_invariance}
 If,  for $i=1, 2$ the function  $\varphi_i$ is injective and bi-measurable, then
    we have:
    \[
      S_f(X_1, X_2)=S_f(\varphi_1(X_1), \varphi_2 (X_2)). 
    \]

 \item\textbf{Invariance for independent Markovian kernels} 
\label{it:Sf-Markov}
Assume, that  for $i=1,2$, the random  variable $X_i$ can be  written as
$X_i=(Y_i, W_i)$ and  takes values in a Polish space  (thus  the
following         probability         kernels        exist).          If
$$
P_{(W_1, W_2)|(Y_1,  Y_2)=(y_1, y_2)}(\rd w_1, \rd  w_2)= \prod_{i=1}^2
P_{W_i|Y_i=y} (\rd w_i)
$$ 
holds $P_{(Y_1, Y_2)}(\rd y_1, \rd y_2)$-a.s., then, we have:
    \[
      S_f\big((Y_1, W_1),\,  (Y_2, W_2)\big)=S_f(Y_1, Y_2). 
    \] 
  \end{enumerate}
\end{prop}

\begin{proof}[Proof of Proposition~\ref{prop:invariance_csiszar_index}]
 Point{\it ~\ref{enum:marginal_invariance}} is a consequence of Proposition~\ref{prop:inv-gen} with $\varphi(x_1, x_2)=(\varphi_1(x_1),\varphi_2(x_2))$. Point{\it ~\ref{enum:bijection_invariance}} is a consequence of Remark~\ref{rem:inv-gen}~\ref{item:inv-gen-sf=f}. 
 Point{\it ~\ref{it:Sf-Markov}} is a direct consequence of
 Remark~\ref{rem:related-work}~\ref{item:div-kernel} with 
 $U=(Y_1, Y_2)$, $V=(W_1, W_2)$, $U'=(Y'_1, Y'_2)$, $V'=(W'_1, W'_2)$
 and $(Y'_1, W'_1)$ and $(Y'_2, W'_2)$ independent and distributed as $(Y_1, W_1)$ and $(Y_2, W_2)$. 
\end{proof}

\begin{rem}[Example of invariance of the  Csisz\'ar index]
  \label{rem:inv-gen-index}
  Let $f\in \F$
  \begin{enumerate}[(i)]
  \item \textbf{Marginal invariance.}
Suppose that $X$ and $Y$ are real valued random variable and that
$(X,Y)$  is
symmetric, that is,  $(-X, -Y)$ and $(X,Y)$ have the same
probability distribution. Then by Proposition~\ref{prop:invariance_csiszar_index}{\it ~\ref{enum:marginal_invariance}}, 
we have 
$$
S_f(|X|,|Y|)=S_f(X,Y).
$$

\item \textbf{Invariance by injection.}
Suppose that $X$ and $Y$ are real valued random variables and that the
cumulative distribution  functions $F_X$ and $F_Y$ are
increasing. By Proposition~\ref{prop:invariance_csiszar_index}{\it ~\ref{enum:bijection_invariance}}, this implies that:
\begin{equation}
   \label{eq:Sf-cop-density}
  S_f(F_X(X), F_Y(Y))=S_f(X,Y).
\end{equation}
If furthermore $F_X$ and $F_Y$ are bijections, then the random variables  $F_X(X)$ and $F_Y(Y)$ are uniformly distributed over $[0,
1]$. We shall give more details in this direction in the next section. 

\item \textbf{``Forgetting'' independent random variables.}
\label{it:removing_independent_variables}
Let $X$, $Y$ and $U$ be random variables with $U$ independent of $Y$ and such that $U$ has a conditional distribution given $X$ (which is satisfied if $U$ takes values in a Polish space).  By Proposition~\ref{prop:invariance_csiszar_index}{\it ~\ref{it:Sf-Markov}}, we deduce that:
    \[
    S_f((X,U), Y)=S_f(X, Y). 
    \]
  \end{enumerate}
\end{rem}

The     next      result     is     a     direct      consequence     of
Proposition~\ref{prop:markov},                  see                 also
\citet[Theorem~7.16]{polyanskiy2025information}.  We say  that the random
variable $(X,Y,Z)$  is a Markov  chain if $X  $ and $Z$  are independent
conditionally on $Y$.

\begin{prop}{\bf (Monotonicity of the Csisz\'ar index along a Markov chain)}
    \label{prop:monoticity_csiszar_index_markov}
    Let  $(X,Y,Z)$ be  a random  variable  such that  $X $  and $Z$  are
    independent  conditionally on  $Y$ and  the conditional  probability
    distributions $P_{Z|Y}$ exists.  For any $f \in \F$, we have:
    \[
    S_f(X,Z) \leq S_f(X,Y) .
    \]
\end{prop}

\subsection{Csiszár index and copulas}
\label{sec:copula}

In this section we prove the Csiszár index  for real-valued random variables can be
written using (some of their) copulas.

For an $\R^d$-valued  random variable $Z=(Z_1, \ldots,  Z_d)$, we denote
its  support (or  the support  of its  probability distribution)  as the
smallest  closed set  $F$ of  $\R^d$ such  that $\P(Z\in  F)=1$ and  its
discrete  support or set of atoms as  $\Delta(Z)=\{z\in  \R^d  \, \colon\,  \P(Z=z)>0\}$
which is an at most countable set.
The cumulative distribution function (cdf) 
of          $Z$,          $F_Z$         is          defined          as:
$F_Z(z)=\P(Z_1\leq    z_1,   \ldots,    Z_d   \leq    z_d)$   for    all
$z=(z_1, \ldots, z_d)\in \R^d$.
When all the one-dimensional marginals
of  $Z$  are uniform  on  $[0,1]$,  then  the  cdf $F_Z$  restricted  to
$[0,1]^d$, which  contains the  support of the  distribution of  $Z$, is
called  a $d$-dimensional  copula, or  simply  copula when  there is  no
ambiguity  on the  dimension  $d$. 

For  an  $\R^d$-valued  random
variable $Z=(Z_1,  \ldots, Z_d)$, we  say that a  $d$-dimensional copula
$C$ is a copula for $Z$ if for all $z_1, \dots ,z_d \in \R$:
\begin{equation}
\label{eq:definition_copula_vector}
    F_Z(z_1, \dots, z_d) = C(F_{Z_1}(z_1) , \dots, F_{Z_d }(z_d)).
\end{equation}
Intuitively a copula  for $Z$ describes the  dependence structure between
its marginals. According to Sklar's theorem there always exists at least
one copula  for $Z$, and this  copula is unique  if and only if  all the
one-dimensional  cdf   $F_{Z_1}\ldots,  F_{Z_d}$  are   continuous.  However, we build in the 
Examples~\ref{exple:unique-cop-rep} and~\ref{exple:different_copula_give_different_index}
below very different copulas for the vector $Z=(X,Y)$ of two Bernoulli random variables.  


The checkerboard copula studied by \cite{lin2025checkerboard} uses a uniform randomization applied simultaneously to all atoms along each coordinate of $Z$, with independent randomizations across coordinates. We generalize this construction by allowing uniform randomizations that may vary across atoms within a given coordinate and may be dependent across different coordinates. We thus define a new family of interpolating copulas.

\begin{defi}{\bf (Interpolating  and checkerboard  copulas)} \\
  \label{defi:interpol-cop}
  Let $Z=(Z_1,  \ldots, Z_d)$  be an  $\R^d$-valued random  variable.
\begin{enumerate}[(i)]
\item \label{item:interpol-cop}  An interpolating  copula for $Z$, denoted by $\Cip$, is the
  cdf of the $[0, 1]^d$-valued random variable $U=(U_1, \ldots, U_d)$ where:
 \begin{equation}
     \label{eq:def-U1-d}
    U_i=F_{Z_i}(Z_i-) + \sum_{x\in \Delta(Z_i)} \alpha_i(x)
   \,  T_{i,x} \ind_{\{Z_i=x\}}
     \quad\text{with}\,\,
     \alpha_i(x)=\P(Z_i=x) 
\end{equation} 
and  $(T_{i,x})_{i\in  \{1, \ldots,  d\}, x\in  \Delta(Z_i)}$ being 
 $[0,1]$-uniform random variables independent of $Z$.

\item The (unique) checkerboard copula for $Z$,  denoted by $\Cb_Z$, is the 
  interpolating  copula where,  for  all $i\in  \{1,  \ldots, d\}$,  the
  random variables  $(T_{i,x})_{x\in \Delta(Z_i)}$  are all  equal, with
  common  value  say $T_i$,  and  the  $[0,1]$-uniform random  variables
  $T_1, \ldots, T_d$ are independent and independent of $Z$.   
\end{enumerate}
\end{defi}

Notice that $U_i$ in~\eqref{eq:def-U1-d} belongs to $ [F_{Z_i}(Z_i-), F_{Z_i}(Z_i)]$,  and thus if the cdf
of $Z_i$ is continuous, then  $U_i=F_{Z_i}(Z_i)$.  In particular, if all
the one-dimensional  cdf $F_{Z_1}\ldots,  F_{Z_d}$ are  continuous (that
is, $\Delta(Z_i)=\emptyset$ for  all $i\in \{1, \ldots,  d\}$), then the
interpolating copula is unique and equal to the copula of $Z$.

For the reader convenience, we provide a short proof of the next lemma  in
Section~\ref{sec:proof_introduction}. 

\begin{lem}[Interpolating copula]
  \label{lem:interpol-cop=cop}
   Let $Z$ be an  $\R^d$-valued  random variable. An interpolating
   copula of $Z$ is a copula for $Z$. 
\end{lem}

\begin{exple}[On the non-uniqueness of the interpolating copula]
  \label{exple:unique-cop-rep}
    Let  $Z=(X,Y)$  be  a  couple  of Bernoulli  random  variables  and
    set $p=\P(X=1)$,  $q=\P(Y=1)  $ and  $r=\P(X=1,Y=1) $. We
    assume:
    \begin{equation}
        \label{eq:pqr}
        p,q\in (0,1) 
        \quad\text{and}\quad
        (p+q-1)_+ < r <
    \min(p,q). 
    \end{equation} 
    Set $\rho=1-p-q+r$. 
    It is easy to check that a $2$-dimensional copula $C$
    is a  copula for $(X,Y)$ if and only if $C(1-p, 1-q) = \rho$.
    The checkerboard copula $\Cb_Z$ 
is given by the cdf of  $(U,V)$
where 
$$
U= (1-p) \ind_{\{X=1\}} + T ( (1-p) \ind_{\{X=0\}} + p \ind_{\{X=1\}} ) 
$$
    and 
    $$
    V=  (1-q) \ind_{\{Y=1\}} + T' ( (1-q) \ind_{\{Y=0\}} + q
    \ind_{\{Y=1\}} ) ,
    $$ where $T$ and $T'$ are independent random variables uniform on
    $[0, 1]$ and independent of $(X,Y)$. 

    We consider the interpolating copula $\Cip_Z$  given by the cdf of $(U, V')$
    with  $V'=  (1-q) \ind_{\{Y=1\}} + (1-T) ( (1-q)
    \ind_{\{Y=0\}} + q \ind_{\{Y=1\}} ) $.

    In Fig.~\ref{fig:comparaison_checkerboard_copula_interpolating_copula}, we display the copulas $\Cb_Z$ and $\Cb_Z-\Cip_Z$ with $p =q=1/2$ and $r = 5/16$. 
 \end{exple}


    \begin{figure}[htbp]
        \centering
        \begin{subfigure}[t]{0.49\textwidth}
            \centering
            \begin{tikzpicture}[scale=0.8]
                \begin{axis}[
                    view={-10}{25}, 
                    xlabel={$u$},
                    xlabel style={font=\large},
                    ylabel={$v$},
                    ylabel style={font=\large},
                    grid=major,
                    domain=0:1,
                    y domain=0:1,
                    samples=\nsamples,
                    samples y=\nsamples,
                    enlargelimits=false,
                    xtick={0,0.5,1},
                    ytick={0,0.5,1},
                    ztick={0,0.5,1,1.5,2},
                    shader=flat,
                ]
                \addplot3[
                    surf, shader=interp
                ]
                (
                    x, y,
                    { (x <= 0.5) && (y <= 0.5) ?
                        (1.25*x*y) :
                      ( (x > 0.5) && (y <= 0.5) ?
                        (0.625*y + 0.75*(x-0.5)*y) :
                        ( (y > 0.5) && (x <= 0.5) ?
                          (0.625*x + 0.75*(y-0.5)*x) :
                          (0.3125 + 0.375*(x-0.5) + 0.375*(y-0.5) + 1.25*(x-0.5)*(y-0.5))
                        )
                      )
                    }
                );
                \end{axis}
            \end{tikzpicture}
            \caption{Plot of the checkerboard copula $\Cb_Z$.}
        \end{subfigure}
        \hfill
        \begin{subfigure}[t]{0.49\textwidth}
            \centering
        \begin{tikzpicture}[scale=0.8]
            \begin{axis}[
                view={-10}{25}, 
                xlabel={$u$},
                xlabel style={font=\large},
                ylabel={$v$},
                ylabel style={font=\large},
                grid=major,
                domain=0:1,
                y domain=0:1,
                samples=\nsamples,
                samples y=\nsamples, 
                enlargelimits=false,
                xtick={0,0.5,1},
                ytick={0,0.5,1},
                shader=flat,
                scaled z ticks=false,
                ]

                \addplot3[
                    surf , shader=interp 
                ]
                (
                    x, y,
                    { (x <= 0.5) && (y <= 0.5) ?
                        (1.25*x*y - (0.3125 * max(0, 2*x + 2*y - 1))) :
                      ( (x > 0.5) && (y <= 0.5) ?
                        (0.625*y + 0.75*(x-0.5)*y - (y*0.3125/0.5 + (0.5-0.3125)*max(0,2*(x-1) + 2*y))) :
                        ( (y > 0.5) && (x <= 0.5) ?
                          (0.625*x + 0.75*(y-0.5)*x - (x*0.3125/0.5 + (0.5-0.3125)*max(0,2*(y-1) + 2*x))) :
                          (0.3125 + 0.375*(x-0.5) + 0.375*(y-0.5) + 1.25*(x-0.5)*(y-0.5) - (0.3125 + (0.5-0.3125)*(2*(y-1) + 1) + (0.5-0.3125)*(2*(x-1) + 1) + 0.3125*max(0,2*(x-1) + 2*(y-1) +1)))
                        )
                      )
                    }
                );
                \end{axis}
        \end{tikzpicture}
            \caption{Plot of the difference  $\Cb_Z-\Cip_Z$.}
        \end{subfigure}
        \caption{Comparison of the copulas $\Cb_Z$ and $\Cip_Z$ in Example~\ref{exple:unique-cop-rep}.}
        \label{fig:comparaison_checkerboard_copula_interpolating_copula}
    \end{figure}
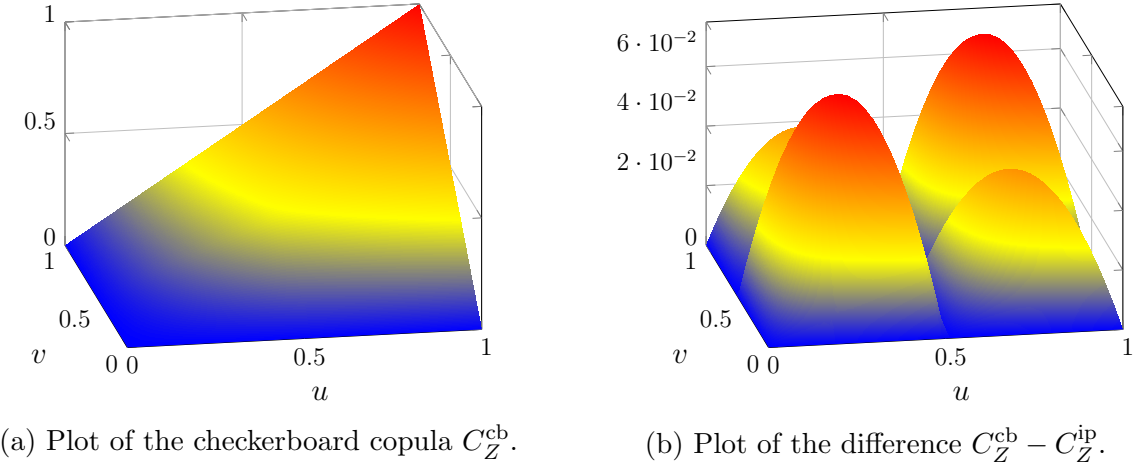

  In  what follows  the  function $(x,y)\mapsto  f(x)  g(y)$ is  denoted
  $f   \otimes    g$.    If   $C_{(X,Y)}$    is   a   copula    of   the
  $\R^{d_X+d_Y}$-valued random variable $(X,  Y)$, then the corresponding
  marginal copulas $C_X$ and $C_Y$ defined by:
   \[
    C_X(x)=C_{(X,Y)}(x,1, \ldots,1)
    \quad\text{for $x\in [0, 1]^{d_X}$}
    \]
    and
    \[
     C_Y(y)=C(1, \ldots, 1, y)
    \quad\text{for $y\in [0, 1]^{d_Y}$}
  \]
 are copulas for $X$ and for $Y$, respectively.

\begin{rem}[Marginal of an interpolating  copula]
  \label{rem:damier-2d}
 Let $X$ and $Y$ be random variables on 
 $\R^{d_X}$ and $\R^{d_Y}$.
 On the one hand,  if  $C_{(X,Y)}$ is an interpolating
  copula  of $(X, Y)$, then the
  marginal copulas $C_X$ and $C_Y$ are    interpolating  copulas  respectively for $X$ and for $Y$.   Notice  also   that  if
  $X$ and $Y$ are independent with  an interpolating copula and
  $C_X$ and $C_Y$, then the copula  $C_X\otimes C_Y$ is an
  interpolating copula for $(X,Y)$. 
  
On the other hand, the marginals $C_X$ and $C_Y$ of 
the checkerboard  copula $C_{(X,Y)}=\Cb_{(X,Y)}$
are  the 
checkerboard  copula of $X$ and of $Y$, that is,
$C_X= \Cb_X$ and $C_Y=\Cb_Y$. 
Furthermore, the random variables $X$ and $Y$ are independent if and
only if: 
\[
  \Cb_{(X,Y)}=\Cb_{X}\otimes \Cb_{Y}.
\]
\end{rem}

We show in the next theorem that the Csiszár  index between $X$
and      $Y$  is always smaller than that  between    their     associated      copulas. Example~\ref{exple:different_copula_give_different_index} shows that the inequality can be strict in the case where the marginals  have atoms. Morever, we show that the interpolating copulas attain the equality. The proof is postponed to Section~\ref{sec:proof_introduction}.  

Recall that if  $C_{(X,Y)}$ is a copula for  $(X,Y)$ then the corresponding marginals  of  $C_{(X,Y)}$,  which for simplicity will be denoted $C_X$ and
$C_Y$,   are copulas for $X$ and
for $Y$. 

\bigskip

\begin{theo}{\bf ($f$-divergence between product of marginal copulas and joint copula)}
\label{thm:interpolating_minimise_fdiv} \\
Let $X=(X_1, \ldots, X_{d_X})$
and $Y=(Y_1, \ldots, Y_{d_Y})$ be real random
vectors of respective dimension $d_X$ and $d_Y$.

\begin{enumerate}[(i)]
   \item \label{it:copu-majoD}
Let $C_{(X,Y)}$ be a copula for  $(X,Y)$. Then we have for $f \in \F$:
\begin{equation}
\label{eq:ineq-cop-Csiszar}
    S_f(X,Y) \leq D_f(C_X \otimes C_Y \Vert C_{(X,Y)}). 
\end{equation}

\item \label{it:copu-equalD} Consider an interpolating copula
  $\Cip_{(X,Y)}$ of $(X,Y)$ given by the cdf of $(U=(U_1, \ldots, U_{d_X}),V=(V_1, \ldots, V_{d_Y}))$ where:
  \begin{align}
\label{eq:def-U}
    U_i&=F_{X_i}(X_i-) + \sum_{x\in \Delta(X_i)} \P(X_i=x) 
         \,  T_{i,x} \ind_{\{X_i=x\}} \quad\text{for $i\in \{1, \ldots, d_X\}$},\\
    \label{eq:def-V}
V_j&= F_{Y_j}(Y_j-) + \sum_{y\in \Delta(Y_j)} \P(Y_j=y) 
     \,  W_{j,y} \ind_{\{Y_j=y\}} \quad\text{for $j\in \{1, \ldots, d_Y\}$},
  \end{align}
  with  $T=(T_{i,x})_{i\in \{1,  \ldots, d_X\},  x\in \Delta(X_i)}$  and
  $W=(W_{j,y})_{j\in  \{1,  \ldots,  d_Y\},  y\in  \Delta(Y_j)}$  having
  uniform one dimensional  marginals on $[0, 1]$, and such  that $T$ and
  $W$  are independent  and  independent of  $(X,Y)$.   Recall that  the
  marginal  copulas   of  $\Cip_{(X,Y)}$  are 
  interpolating copulas for $X$ and for $Y$, $\Cip_X$ and $\Cip_Y$ respectively, as well as  cdf of $U$ and of
  $V$.  Then we have for $f \in \F$:
\[
        S_f(X,Y) = D_f(\Cip_X \otimes \Cip_Y \, \Vert\,  \Cip_{(X,Y)}). 
\]

\item \label{it:copu-cb}
  In particular, considering the checkerboard copulas for $X$, $Y$ and $(X,Y)$ respectively, we have:
\begin{equation}
   \label{eq:S_f_copula}
    S_f(X,Y) = D_f(\Cb_X \otimes \Cb_Y \, \Vert\,  \Cb_{(X,Y)})
    = \min D_f(C_X \otimes C_Y \Vert C_{(X,Y)}),
  \end{equation}  
  where the minimum is taken over all the copulas $C_{(X,Y)}$ of $(X,
  Y)$. 
\end{enumerate}
\end{theo}

\begin{rem}[The two dimensional case]
  \label{rem:2d}
Let  $X$ and $Y$ be  real-valued random variables. Since   the  copulas $C_X$ for
the univariate random variable  $X$ and $C_Y$ for the univariate $Y$ are the identity map on
  $[0,1]$,  equation~\eqref{eq:S_f_copula} can be rewritten as:
    \begin{equation} 
    \label{eq:S_f_copula-d=2}
        S_f(X,Y) =   D_f(\Pi \, \Vert\,  \Cb_{(X,Y)})
        = \min D_f(\Pi  \Vert C_{(X,Y)}),
    \end{equation} 
    where the copula $\Pi(u,v)=uv$ for  $u,v\in [0,1]$ corresponds to the
    cdf of two independent uniform random variables on $[0, 1]$, and the
    infimum is taken over all the copula $C_{(X,Y)}$ for $(X,Y)$. Notice
    that the checkerboard copula $\Cb_{(X,Y)}$ could have been replaced
    by any  interpolating copula $\Cip_{(X,Y)}$. 
\end{rem}

\begin{rem}[Invariance by increasing transformation] 
  \label{rem:cop-smooth}
Let  $X$ and $Y$ be  real-valued random variables  for simplicity. 
From the expression~\eqref{eq:S_f_copula} of $S_f(X,Y)$ in terms of a
copula, we  deduce that the Csiszár  index is invariant 
when $X$ or $Y$ are transformed by  strictly increasing functions. This
can also be deduced from
Proposition~\ref{prop:invariance_csiszar_index}~{\it \ref{enum:bijection_invariance}}.
\end{rem}

If all coordinates of $X$ and $Y$ have atom-free distributions, {\it i.e.} they have continuous cdf and  $\Delta(X_i)=\Delta(Y_j)=\emptyset$ for all $i,j$, then the copula for $(X,Y)$ is unique and~\eqref{eq:ineq-cop-Csiszar} is thus an equality. 

\begin{exple}[Inequality~\eqref{eq:ineq-cop-Csiszar} might be strict] 
  \label{exple:different_copula_give_different_index}
  Let  $(X,Y)$  be  a  couple  of Bernoulli  random  variables  and  set $p=\P(X=1)$,  $q=\P(Y=1)  $ and  $r=\P(X=1,Y=1) $ with $p$, $q$ and $r$ satifying~\eqref{eq:pqr}. 
Set:
  \[
    \theta = \frac{r-pq}{pq(1-p)(1-q)},
  \]
and assume that $\theta$ belongs to  $ [-1, 1]$. Then, 
the 
Farlie-Gumbel-Morgenstern copula $C_\theta$ with density 
 $    \partial^2_{u,v} C_\theta (u,c)=c_\theta(u,v) = 1+  \theta (1-2u)(1-2v)$  
 for $u,v\in [0, 1]$ is a copula for $(X,Y)$ as 
$C_\theta(1-p, 1-q) = 1-p-q+r$.
    Then, using that for $f\in \F$:
    \begin{align*}
      S_f(X,Y) &= f\left( \frac{(1-p)(1-q)}{1-p-q+r}\right) (1-p-q+r) 
      + f\left( \frac{p(1-q)}{p-r}\right) (p-r) \\
      &+ f\left( \frac{q(1-p)}{q-r}\right) (q-r) + f\left( \frac{pq}{r}\right) r, 
    \end{align*} 
we get for the Pearson $\chi^2$ divergence (that is, $f_\mathrm{P}(t)=t^2-1$)
that: 
\[
D_\mathrm{P}(\Pi \, \Vert\,  \Cb) =S_{f_\mathrm{P}}(X,Y)=  \frac{\left(p q -
        r\right)^{2} \left(p^{2} q + p q^{2} - 2 p q r - p q +
        r^{2}\right)}{r \left(p - r\right) \left(q - r\right) \left(p +
        q - r - 1\right)}  ,
\]
 with $\Pi$ the independent copula
 and $\Cb$ the  checkerboard copula for $(X,Y)$ as well as:
\[
   D_\mathrm{P}(\Pi \Vert C_\theta) = \inv{2 \theta} \left(\mathrm{Li}_2(\theta) - \mathrm{Li}_2(-\theta) \right) - 1,
\]
with  $\mathrm{Li}_2(\theta)=\sum_{k\in \N^*} \theta^k/k^2$  the dilogarithm function. 

The densities of the corresponding  checkerboard copula and 
Farlie-Gumbel-Morgenstern are represented in  Fig.~\ref{fig:density_interpolating_and_FGM_copula} for the particular case $p=q=1/2$ and $r=5/16$ and thus $\theta = 1$; in this case, we have $D_\mathrm{P}(\Pi \, \Vert\,  \Cb) =S_{f_\mathrm{P}} (X,Y)=1/15< (\pi^2/8) -1=D_\mathrm{P}(\Pi \, \Vert\,  C_\theta)$, which proves that the inequality~\eqref{eq:ineq-cop-Csiszar} is strict for this example. 
\end{exple}

  \begin{figure}[htbp]
    \centering
    \begin{subfigure}[t]{0.49\textwidth}
        \centering
        \begin{tikzpicture}[scale=0.8]
            \begin{axis}[
                view={-10}{25}, 
                xlabel={$u$},
                xlabel style={font=\large},
                ylabel={$v$},
                ylabel style={font=\large},
                grid=major,
                domain=0:1,
                y domain=0:1,
                zmin=0.5,      
                zmax=1.5, 
                samples=\nsamples,
                samples y=\nsamples,
                enlargelimits=false,
                xtick={0,0.5,1},
                ytick={0,0.5,1},
                ztick={0,0.5,1,1.5},
                shader=flat,
            ]

            \addplot3[
                patch,
                patch type=bilinear,
                shader=flat,
            ] table[meta=meta] {
                x y z meta
                0   0   1.25   1.25
                0.5 0   1.25   1.25
                0.5 0.5 1.25   1.25
                0   0.5 1.25   1.25

                0.5 0.5 1.25   1.25
                1   0.5 1.25   1.25
                1   1   1.25   1.25
                0.5 1   1.25   1.25

                0   0.5 0.75   0.75
                0.5 0.5 0.75   0.75
                0.5 1   0.75   0.75
                0   1   0.75   0.75

                0.5 0   0.75   0.75
                1   0   0.75   0.75
                1   0.5 0.75   0.75
                0.5 0.5 0.75   0.75
            };

            \end{axis}
        \end{tikzpicture}
        \caption{Density of the checkerboard copula $\Cb$.}
        \label{fig:density_interpolating_copula}
    \end{subfigure}
    \hfill
    \begin{subfigure}[t]{0.49\textwidth}
        \centering
        \begin{tikzpicture}[scale=0.8]
            \begin{axis}[
                view={-10}{25}, 
                xlabel={$u$},
                xlabel style={font=\large},
                ylabel={$v$},
                ylabel style={font=\large},
                grid=major,
                domain=0:1,
                y domain=0:1,
                samples=\nsamples,
                samples y=\nsamples,
                enlargelimits=false,
                xtick={0,0.5,1},
                ytick={0,0.5,1},
                ztick={0,0.5,1,1.5,2},
                shader=flat,
            ]

            \addplot3[surf, shader=interp
            ] 
            {1 + (1-2*x)*(1-2*y)};
            \end{axis}
        \end{tikzpicture}
        \caption{Density of the Farlie-Gumbel-Morgenstern copula $C_\theta$ (with $\theta = 1$).}
        \label{fig:density_FGM_copula}
    \end{subfigure}
    \caption{Two copulas for the discrete random variable $(X,Y)$ with Bernoulli marginals from Example~\ref{exple:different_copula_give_different_index}.}
    \label{fig:density_interpolating_and_FGM_copula}
\end{figure}
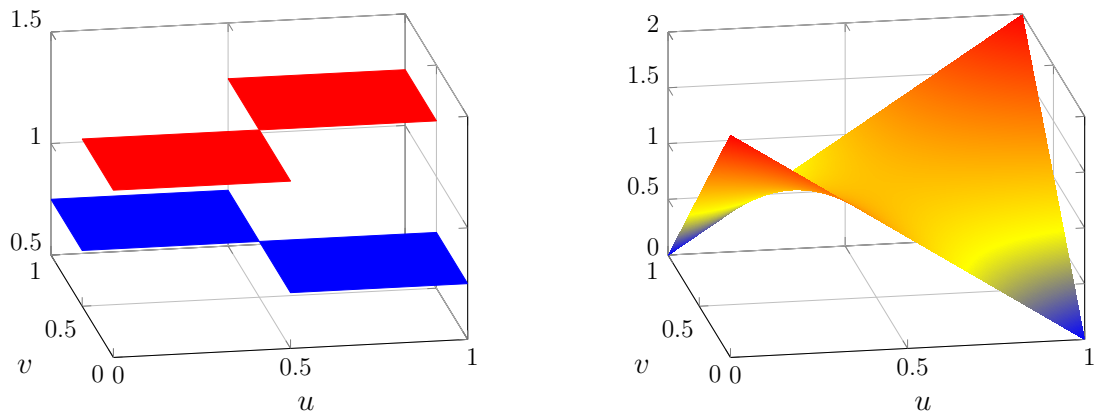 

\clearpage

\bibliographystyle{spbasic}
\bibliography{references}

\clearpage

\section{Additional results and proofs}
\label{sec:proof_introduction}

\subsection{On convex functions}

We give a short proof of the Jensen inequality which covers a more general setup.
Indeed, we do not assume here that $\varphi$ is lower
    semi-continuous; take for example:
    $$
    \varphi(x)=\infty  \ind_{[1, 2]^c}(x) +
    \ind_{\{1,2\}}(x).
    $$
    In this case, it is not true that $\varphi=\sup \ell$, where
    the supremum is taken over all the affine functions $\ell$ such that
    $\varphi\geq  \ell$ and the latter property  is the principal ingredient of a
    classical proof of Jensen inequality.

\begin{lem}[Jensen's inequality]
\label{lem:equality_jensen}
Let X be an $\bar \R$-valued  random variable. Let $\varphi$ be a proper
convex function defined on $\R$ (and by continuity on $\bar \R$). Assume
that $X$  and $\varphi(X)$ are  quasi-integrable (this is  in particular
the case if $X$ is integrable). Then, we have:
\begin{equation}
  \label{eq:def-proper-cvx-E}
 \varphi(\E[X]) \leq \E[\varphi(X)].
\end{equation}

Furthermore if $X$ is integrable and $\varphi$ is strictly convex at
$\E[X]$, then the inequality~\eqref{eq:def-proper-cvx-E} is an equality
if and only if a.s.\ $X =\E[X]$.
\end{lem}

\begin{rem}
  \label{rem:jensen}
  Notice   that  $X$   integrable  implies   that  $\varphi(X)$   is
  quasi-integrable   as  $\varphi$   is  convex;   but  $\varphi(X)$
  quasi-integrable does not imply that $X$ is quasi-integrable (take
  $\varphi(x)=|x|$    and    $\E[X_+]=\E[X_-]=\infty$);   and    $X$
  quasi-integrable   does    not   imply either  that    $\varphi(X)$   is
  quasi-integrable (take $\varphi(x)=x+ x_+^2$,
  $\E[X_-]=\E[X^2_+]=+\infty$ and  $\E[X_+]<\infty$).
\end{rem}

\begin{proof}[Proof of Lemma~\ref{lem:equality_jensen}]
  First notice that to prove~\eqref{eq:def-proper-cvx-E}, it is
  enough to  consider the case  $\E[\varphi(X)]\in [-\infty , + \infty )$
  and   $X$   non  constant.

  We then claim  that $\varphi$ is continuous  at $\E[X]$.  Indeed,
  since this is  the case by convention if $\E[X]=\pm\infty  $, we shall
  only consider the case $X$ integrable  and thus a.s.\ $X\in \R$. Then,
  since a.s.\ $\varphi(X)<+\infty $, we deduce that a.s.\ $X$ belongs to
  $\dom(\varphi)$. The latter being convex, $X$ non constant and $\E[X]$
  finite,  we   deduce  that   $\E[X]$  belongs   to  the   interior  of
  $\dom(\varphi)$. As the function $\varphi$ is convex, it is continuous
  in  the interior  of  its  domain. In conclusion,  the  function $\varphi$  is
  continuous at $\E[X]\in \bar \R$.

 Let $(X_n, n\in \N^*)$ be independent random
  variables distributed as $X$. Set $\bar X_n=n^{-1} \sum_{k=1}^n X_k$
  for $n\in \N^*$.
  Since $X$ and $\varphi(X)$ are quasi-integrable, the
  strong law of large number gives that a.s.\ 
  $$
  \lim_{n\rightarrow \infty
  } \bar X_n= \E[X] \text{ and }\lim_{n\rightarrow \infty
  } n^{-1} \sum_{k=1}^n \varphi(X_k)= \E[\varphi(X)].
  $$
Since $\varphi$ is convex, we also get that $\varphi(\bar X_n) \leq
n^{-1} \sum_{k=1}^n \varphi(X_k)$. Since $\varphi$ is continuous at
$\E[X]$,
we deduce that a.s.\ $\lim_{n\rightarrow \infty
} \varphi(\bar X_n) = \varphi(\E[X])$;
and thus~\eqref{eq:def-proper-cvx-E}
holds.

Assume  that $X$  is  integrable  and $\varphi$  is  strictly convex  at
$\E[X]\in \R$.  This implies that
$\E[X]\in \dom (\varphi)$  and as~\eqref{eq:def-proper-cvx-E} is an
equality, we deduce that a.s.\ $X \in \dom  (\varphi)$.  We shall
assume  that $X$  is  not  constant. Since  $X$  is  integrable and  non
constant,  this  implies  that  $\E[X]$   belongs  to  the  interior  of
$\dom ( \varphi)$.   Thus,  there  exists an  affine  function  $g$  which
coincides with $\varphi$  at $\E[X]$ and such that $\varphi  \geq g$ and
$\varphi(t)>g(t)$ for all $t>\E[X]$ or  for all $t<\E[X]$.  Without loss
of  generality, we  can  assume that  $\varphi>g$  on $A=\{t>  \E[X]\}$.
Since   $X$  is   not  constant,   we  get   $\P(X\in  A)>0$   and  thus
$\E[\varphi(X)]=        \E[\varphi(X)\ind_{\{X\in         A\}}]        +
\E[\varphi(X)\ind_{\{X\not \in A\}}]> \E[g(X)]=g(\E[X])=\varphi(\E[X])$.
So if~\eqref{eq:def-proper-cvx-E}  is an equality, we  deduce that a.s.\
$X=\E[X]$. \hfill \qed
\end{proof}

We now give an
elementary lemma on convex functions. Recall that a function $\varphi\in
\F$ might not be continuous at the boundary of its domain.
Recall the convention $yf^*(0)=0$ if $y=0$.
\begin{lem}[Property of convex functions]
\label{lem:cvx-prop}
Let $f\in \F$ and $z\in \R_+$. We have for all $x\in [0,z]$:
\begin{equation}
   \label{eq:ineg-str-cvx}
   f(z) \leq   f(x) + (z-x) f^*(0).
\end{equation}
  Furthermore, if $f$ is strictly convex at $z>0$ the
  inequality~\eqref{eq:ineg-str-cvx} is strict for all $x\in [0, z)$.
\end{lem}
\begin{proof}
  The  result is  clear if  $x=z$ or  $x\not\in \dom  (f)$ or  $x<z$ and
  $f^*(0)=+\infty  $.   We  assume   that  $x\in  \dom(f)$,   $x<z$  and
  $f^*(0)<+\infty    $.    The    latter    condition    implies    that
  $\sup \dom(f)=+\infty  $.
  Set $y = z-x \in \R_+^*$.
  For  $\varepsilon\in (0,  1/y)$, we  get by
  convexity:
   \[
f(x+y) \leq  (1-\varepsilon y) f\left(\frac{x}{1- \varepsilon y}\right)
+ y \varepsilon f \left(\inv{\varepsilon}\right).
\]
Notice that  $f$ is continuous at $x$ if $x$ belongs to the interior of
the domain of $f$ and it has a right limit, $f(x+)$ at $x=
\inf \dom(f)$, and that in both cases $f(x+)\leq  f(x)$.
Then, letting $\varepsilon$ goes down to $0$, we get
$  f(x+y) \leq  f(x+) + yf^*(0)\leq  f(x) + y f^*(0)$.

Suppose that  $f$ is strictly convex at $z>0$. Let $x\in [0, z)$,
$\alpha\in (0, 1)$ and $y>z$ such that $z=\alpha x + (1-\alpha)
y$. Since $f$ is strictly convex at $z$, we get $\alpha(f(z)-f(x)) <
(1-\alpha) ( f(y) -f(z))$, that is:
\[
  \frac{f(z)- f(x)}{z-x} < \frac{f(y) - f(z)}{y-z}\leq  f^*(0),
\]
where we used~\eqref{eq:ineg-str-cvx} (with $z$ and $x$ replaced by $y$
and $z$) and $f(z)$ finite for the last inequality.
\end{proof}


\subsection{Proofs of results in Section~\ref{sec:csiszar_divergence}}

\begin{proof}[Proof of Lemma~\ref{lem:f_div_properties}]
  \label{p:Df-prop}
 {\it ~\ref{enum:f_div_prop_nonnegativity}} Let $h= \rd P/ \rd Q$, which is defined $(P+Q)$-a.e.. Set $x=
  Q(h)=1- P(h=+\infty )=1- P(q=0)\in [0, 1]$.
Using~\eqref{eq:definition_df0},  Jensen's inequality and then
Lemma~\ref{lem:cvx-prop} with $z=1$, we get:
\begin{equation}
\label{eq:proof_df_positive_inegalite1}
 D_f(P\Vert Q) = Q(f(h))+ f^*(0) P(q=0) \geq   f\left( x \right) + (1-x) f^{*}(0) \geq f(1) = 0.
\end{equation}
This proves the  nonnegativity of
the $f$-divergence.

{\it \ref{enum:f_div_prop_reflexivity}} We now assume that $f$ is strictly convex at 1 and $D_f(P\Vert Q)=0$, so
all the inequalities in~\eqref{eq:proof_df_positive_inegalite1} are
equalities. By Lemma~\ref{lem:cvx-prop} the first one implies that
$x=1$; by Jensen inequality, as $Q(h)=x=1$, the first one implies that $h=1$ $Q$-a.s.,
that is $P=Q$ by~\eqref{eq:Leb-decomp}.

\medskip

{\it \ref{enum:div_prop_supremum}} We     now    give     the    supremum  of the  $f$-divergence.   The functions  $f$  and $f^*$  being
convex   and  zero   at  $1$,   we  have   $f(x)  \leq   (1-x)f(0)$  and
$f^*(x)   \leq    (1-x)f^*(0)$   for   $x   \in    [0,1]$.   We   deduce
from~\eqref{eq:decom_Df} that:
\begin{eqnarray}
\label{eq:upper_bound_Df}
D_f(P\Vert Q) &\leq & f(0) \int_{0\leq  p<q} (q-p) \, \rd \lambda + f^*(0)
\int_{0\leq  q<p}(p-q)  \, \rd \lambda \nonumber \\
&= &(f(0)+f^*(0)) \int_\Omega (q-p)_+ \, \rd \lambda.
\end{eqnarray}
Since the last integral is  belongs to [0,1], we have that $D_f(P \Vert
Q) \leq f(0) + f^*(0)$. Furthermore, if we take $P$ and $Q$ mutually
singular, we directly get  $D_f(P\Vert Q) = f(0) + f^*(0)$ from the definition. This
gives{\it ~\ref{enum:div_prop_supremum}}.

{\it \ref{enum:div_prop_condition_sup}} Assume $D_f(P\Vert Q) = (f + f^*)(0)$ and
$(f+f^*)(0)$ is positive and  finite.  We deduce that the last integral
in~\eqref{eq:upper_bound_Df} has to be equal to 1, which means that $P$ and $Q$
are singular. 

\medskip

The    duality    formula{\it ~\ref{enum:f_div_prop_duality}}  is    obvious
from~\eqref{eq:decom_Df}.  For    the  invariance
property{\it ~\ref{enum:f_div_prop_invariance}}, notice  that for $g=  f+ f_c$
with   $f_c\in  \F$   and  $f_c=c   (\id   -1)$  on   $\R_+$,  we   have
$D_g=D_f+ D_{f_c}=D_f$ thanks to Remark~\ref{rem:affine}.

\medskip

{\it ~\ref{enum:f_div_prop_symmetry}} We   follow~\cite{osterreicher2002csiszar}   for   the  proof   of   the
symmetry{\it ~\ref{enum:f_div_prop_symmetry}}.                             Let
$A\in    \cf\setminus\{\emptyset,    \Omega\}$     and    assume    that
$  D_f(P \Vert  Q) =  D_{f}(Q  \Vert P)  $  for all  $P,Q\in \cp$.   Let
$x\in  A$  and  $y\in  A^c$,   and  consider  the  probability  measures
$P_t=t \delta_x +  (1-t)\delta_y$ for $t\in [0, 1]$.  For $s,t\in (0,1)$
and $u=t/s$, we have:
\begin{eqnarray}
   \label{eq:Df=B(s,t)}
   D_f(P_s \Vert P_t) &=& f(s/t) t + f((1-s)/(1-t)) (1-t) \nonumber \\
   &=& s f^*(u)+ (1-su) f \left( \frac{1-s}{1-su}\right).
 \end{eqnarray}
 Since  by assumption
 $  D_f(P_s \Vert P_t)=   D_f(P_t \Vert P_s)$,
we obtain that for $s\in (0, 1)$ and $u\in \R^*_+$ such that $su<1$:
 \[
   s f^*(u)+ (1-su) f \left( \frac{1-s}{1-su}\right)
   =  s f(u)+ (1-su) f^* \left( \frac{1-s}{1-su}\right).
 \]
 Letting $s$ go down to $0$, and using that $1$ belongs to the
 interior of the domains of $f$ and $f^*$, we deduce that $f(u)=+\infty
 $ if and only if $f^*(u)=+\infty $. This gives that
 $\dom(f^*)=\dom(f)$.
 For $u\in \dom(f)$ and $su<1$ with $s$ small enough so that
 $(1-s)/(1-su)\in \dom(f)$, we get that:
 \[
   (f^*-f)(u)=(f^* -f) \left( \frac{1-s}{1-su}\right)
   \frac{1-su}{s}\cdot
 \]
Since $f^*(1) = f(1)$, we can further write
 \[
   (f^*-f)(u)= \frac{(f^*-f)((1-s)/(1-su)) - (f^*-f)(1)}{(1-s)/(1-su) - 1 }\cdot (u-1).
\] 
Notice that $f$ and  $f^*$ have a finite right and left  derivative at $1$ and
 that,  as $(f^*)'(1+)=-f'(1-)$,  $(f^*-f)'(1+)= (f^* -f)'(1-)$. Denote
 by $2c$ this latter quantity. 

 Letting
 $s$ and $t$ go to $0$ with  $u=t/s\in \dom(f)$ fixed, we deduce that
 $ (f^*-f)(u)=2c(u-1)$. Recall that $\dom(f)=\dom(f^*)$. Set $g=f + c(\id-1)$ to get  $g\in \F$ and
 $g^*=g$. Hence, taking $h=g/2$ gives that $h\in \F$ and
 $f=h+h^* + c(\id-1)$. 
 
 On the other hand, if $h\in \F$ and $c\in \R$, setting
$f=h+h^* + c(\id-1)$, we get that $\dom(f)=\dom(f^*)$ and
 $f^* -f = 2 c(\id-1)$ on $\dom(f)$. Then, use the invariance property{\it ~\ref{enum:f_div_prop_invariance}} to get that $D_f(P  \Vert Q)
 = D_{f}(Q\Vert  P)$ for all $P,Q\in  \cp$.
 This finishes the proof of{\it ~\ref{enum:f_div_prop_symmetry}}.

{\it \ref{enum:f_div_prop_range}} Let $\crr$ denote the range of possible values of $D_f$. Assume that $f$ is finite on
$\R_+^*$.
Letting $s$ go from 0 to 1 in~\eqref{eq:Df=B(s,t)} with $u\in (0, 1)$
fixed, we get by continuity of $f$ that $(0, f^* (u) + (1-u) f(0))\subset \crr$
and then, letting $u$ go down to $0$,  that $(0, (f^*+f)(0))\subset \crr$.
Then use{\it ~\ref{enum:f_div_prop_nonnegativity}}
and{\it ~\ref{enum:div_prop_supremum}} to deduce that
$\crr=[0,(f^*+f)(0)]$. This gives{\it ~\ref{enum:f_div_prop_range}}. 
\end{proof}

\begin{rem}
  \label{rem:PL-ext}
  Let $(\Omega, \cf, \lambda)$ be a non-trivial $\sigma$-finite measured
  space: there  exist $A,  B\in \cf$ disjoint  and with  positive finite
  measures.    Using    the    probability    distributions
  $$
  P_s=\left(s \lambda(A)^{-1} \ind _A + (1-s) \lambda(B)^{-1} \ind _{B}
  \right)\,  \lambda,
  $$ we  easily get~\eqref{eq:Df=B(s,t)}  and can  thus
  adapt   the   proof  above  to    deduce that the
  properties{\it ~\ref{enum:f_div_prop_symmetry}}
  and{\it ~\ref{enum:f_div_prop_range}}  hold  also  with  $\cp$  replaced  by
  $\cp_\lambda\subset\cp$ the  subset of probability measures  which are
  dominated by $\lambda$.
\end{rem}

\begin{proof}[Proof of Proposition~\ref{prop:inv-gen}]
  Set   $h=   \rd    P/\rd   Q$,
    $P' = \varphi_ \sharp P $  and $Q' = \varphi_ \sharp Q$
   for   simplicity.     Since   $h$   is
  $\sigma(\varphi)$-measurable,  there exists  a nonnegative  measurable
  function  $h'$  defined on  $E$  such  that $h=h'\circ  \varphi$.   By
  construction,  we have  for  any nonnegative  measurable function  $g$
  defined on $\Omega$:
\[
  P'(g)=  P(g\circ \varphi)
  =Q(h\, g\circ \varphi) + P( g\circ \varphi \, \ind_{\{h=\infty \}})
  =  Q' (h' \, g) + P'(g \,  \, \ind_{\{h'=\infty \}}).
\]
This gives $\rd P'= h'\, \rd Q' + \ind_{\{h'=\infty \}}\rd P'$ and, by
symmetry, we deduce that~\eqref{eq:Leb-decomp} holds with $\nu=P'$ and
$\mu=Q'$. By uniqueness of the Lebesgue decomposition, we get that  $
\rd P'/ \rd Q'=h'$ (with the
convention~\eqref{eq:Leb-dd2}). This proves the first part of the lemma.

Let $f\in \F$. Using~\eqref{eq:chgt-var}, we get that:
\begin{align*}
D_f( P' \, \Vert \,  Q')
 & = Q'(f(h'))+ f^*(0) \, P' \left( h' =\infty \right)\\
&  = Q(f(h)) + f^*(0) \, P \left( h =\infty \right)
= D_f( P \, \Vert \,  Q).
\end{align*}
This concludes the proof. 
\end{proof}

\begin{proof}[Proof of Theorem~\ref{thm:increase-Df}]
For simplicity, we write  $P'=\varphi_\sharp P$ and $Q'=\varphi_\sharp
Q$, and consider the probability measures
$\lambda=(P+Q)/2$ and $\lambda'=\varphi_\sharp \lambda$, that is,
$\lambda'=(P'+Q')/2$. We denote by $p$  (resp.\ $p'$)
the density of $P$  (resp.\ $P'$ ) w.r.t.\ $\lambda$
(resp.\ $\lambda'$), and similarly for $q$ and $q'$.
 Set  $\tilde p=p'\circ \varphi$ and define $\tilde q$ similarly. As:
\[
     \E_\lambda[p'\circ \varphi\, j\circ \varphi]
    =\E_{\lambda'}[p' \, j]=\E_{P'}[j]=\E_P[j\circ
    \varphi]=\E_\lambda[p\, j\circ
    \varphi]=\E_\lambda[\E_\lambda[p\, |\, \varphi] \, j\circ \varphi],
    \]
 with $j$ any nonnegative measurable function defined on $E$, we deduce
 that $\lambda$-a.s.:
\[
  \tilde p =\E_\lambda[p\,
  |\, \varphi]
  \quad\text{and}\quad
  \tilde q =\E_\lambda[q\,
  |\, \varphi].
\]

Since $\E_\lambda[p \ind_{\{\tilde p =0\}}]= \E_\lambda[\tilde p
\ind_{\{\tilde p =0\}}]= 0$, we deduce that:
\begin{equation}
   \label{eq:p-tp=0}
  \ind_{\{p>0, \tilde p=0\}}=0
  \quad\text{$\lambda$-a.s.}
  \quad\text{as well as}\quad
  \tilde p>0 \quad P\text{-a.s.,}
\end{equation}
and similarly for $q$.
We set:
\[
  h_+= \E_\lambda[ p\,  \ind_{\{q>0\}} \, |\, \varphi]
  \quad\text{and}\quad
  h_0= \E_\lambda[ p\,  \ind_{\{q=0\}} \, |\, \varphi],
\]
so that $\tilde p= h_+ + h_0$. Using~\eqref{eq:p-tp=0} with $q$ instead of $p$, we get that $h_+=0$ on
$\{\tilde q=0\}$ and we set:
\[
  g_+=\frac{h_+}{\tilde q},
\]
which by convention is $0$ on $\{\tilde q=0\}$, and similarly, as $h_0=0$ on
$\{ \tilde p=0\}$:
\[
  g_0= \frac{h_0}{\tilde p},
\]
which by convention is $0$ on $\{\tilde p=0\}$.
We claim that:
\begin{equation}
   \label{eq:g+g0}
  g_+ = \E_Q\left[\frac{\rd P}{\rd Q} \, \big |\, \varphi\right]
  \quad Q\text{-a.s.}
  \quad\text{and}\quad
  g_0= \P\left(\frac{\rd P}{\rd Q}=+\infty  \, \big |\, \varphi\right)
  \quad P\text{-a.s.},
\end{equation}
where by convention $P_\mu(A\, |\, \varphi)=\E_\mu[\ind_A \,|\, \varphi]$.
Indeed, the representation of $g_+$ (which is
$\sigma(\varphi)$-measurable) is a consequence of:
\[
  \E_Q\left[ \frac{\rd P}{\rd Q} \, j \circ \varphi\right]
  = \E_\lambda [p \ind_{\{q>0\}}\, j\circ \varphi]
  = \E_\lambda[h_+ \, j\circ \varphi]
  = \E_\lambda[q \, g_+ \, j\circ \varphi]
  = \E_Q[ g_+ \, j\circ \varphi],
\]
with $j$ any nonnegative measurable function defined on $E$.
The proof of the representation of $g_0$ is similar (use  $
  \E_P\left[\ind_{\{ \frac{\rd P}{\rd Q}=+\infty \}} \, j \circ \varphi\right]
  = \E_\lambda [p \ind_{\{q=0\}}\, j\circ \varphi]
  = \E_\lambda[h_0 \, j\circ \varphi]
  = \E_\lambda[p \, g_0 \, j\circ \varphi]
  = \E_P[ g_0 \, j\circ \varphi]$,
with $j$ any nonnegative measurable function defined on $E$).
 We deduce   that:
\begin{equation}
   \label{eq:P'/Q'}
   \frac{\rd P'}{\rd Q'}\circ \varphi
   = \frac{\tilde p}{\tilde q}
     = g_+  + \frac{\tilde p }{\tilde q} g_0
   = g_+ \ind_{\{\tilde q >0\}} + \frac{\tilde p }{\tilde q} g_0
   \ind_{\{\tilde p >0\}}.
 \end{equation}

 The interesting contribution to $\rd P'/\rd Q'$ is on $\{ \tilde p\tilde  q >0\}$. We need to distinguish two sub-case. Set:
 \[
   \widetilde {pq} =\E_\lambda [pq\, |\, \varphi].
 \]
 On $\{ \widetilde {pq}=0, \, \tilde p \tilde q >0\}$ we get $h_+=0$ and
 thus  $g_+=0$   and  $g_0=1$,   and  there   is  no   more  information
 which can be deduced from~\eqref{eq:P'/Q'}.
 On $\{ \widetilde {pq}>0\}$, we get that $\E_\lambda[ p\,
 \ind_{\{q>0\}} \, |\, \varphi]>0$ and thus $h_0< \tilde p$ and $g_0<1$
 (as $\tilde p>0$). This gives that:
 \[
   \frac{\tilde p}{\tilde q}= \frac{g_+}{1- g_0}\in (0, +\infty ).
 \]
 This can be rewritten as:
 \[
   \frac{\tilde p}{\tilde q}=
   \frac{\E_Q\left[\frac{\rd P}{\rd Q} \, \big |\, \varphi\right]}
   {\P\left(\frac{\rd P}{\rd Q}<+\infty  \, \big |\, \varphi\right)}
   \in (0, +\infty )
   \quad\text{on}\quad
  \{ \widetilde {pq}>0\}=  \{ \widetilde {pq} \tilde p \tilde q>0\}.
 \]

 \medskip

In a second step
we
write  $  D_f (P' \, \Vert\, Q')=A+ f^*(0) \,  B$ with:
\[
  A= \E_{\lambda'} \left[ f\left(\frac{p'}{q'} \right)
    q'\right]
  =\E_{\lambda} \left[ f\left(\frac{\tilde p}{\tilde q} \right)
    \tilde q\right]
    \]
and
\[
B=\E_{\lambda'} [p'\, \ind_{\{q'=0\}}]
=\E_{\lambda} [\tilde p\, \ind_{\{\tilde q=0\}}]
=\E_{\lambda} [ p\, \ind_{\{\tilde q=0\}}].
\]
We have  that:
\begin{align}
  \label{eq:A1-majo1}
  A= \E_{\lambda} \left[ f\left(\frac{\tilde p}{\tilde q} \right)
     q\right]
  = \E_Q\left[ f\left(g_+ + \frac{h_0 }{\tilde q}  \right)\right]
 & \leq  \E_Q[f(g_+)] + f^*(0)\, C\\
  \label{eq:A1-majo2}
 & \leq  \E_Q\left[f\left(\frac{\rd P}{\rd Q}\right)\right] +
   f^*(0)\,C,
 \end{align}
where we used  Lemma~\ref{lem:cvx-prop} for the first inequality and
Jensen inequality for the second and:
\begin{equation}
   \label{eq:C}
  C=  \E_Q \left[\frac{h_0 }{\tilde q}  \right]
  = \E_\lambda \left[ q \frac{h_0}{\tilde q}  \right]
  = \E_\lambda \left[ \tilde q  \frac{h_0}{\tilde q}  \right]
  =\E_\lambda \left[ p \ind_{\{q=0, \, \tilde q>0\}}  \right].
\end{equation}
Notice that, thanks to~\eqref{eq:p-tp=0} with $p$ replaced by $q$:
\begin{align*}
  B+C &= \E_\lambda \left[ p \ind_{\{q=0, \, \tilde q=0\}}  \right]
  + \E_\lambda \left[ p \ind_{\{q=0, \, \tilde q>0\}}  \right] \\
  &=\E_\lambda \left[ p \ind_{\{q=0\}}  \right]
  =P( \rd P/\rd Q=\infty ).    
\end{align*}
We thus deduce that:
\[
   D_f (P' \, \Vert\, Q')=A+ f^*(0) \,  B \leq   D_f (P \, \Vert\,
   Q).
 \]

 \medskip

 We                   now                  assume                   that
 $ D_f (P' \, \Vert\, Q')= D_f  (P \, \Vert\, Q)<+\infty $. This implies
 that  the inequalities~\eqref{eq:A1-majo1}  and~\eqref{eq:A1-majo2} are
 equalities.  Since the last  term in the inequality~\eqref{eq:A1-majo1}
 is finite (as $D_f (P \, \Vert\, Q)<+\infty $) we deduce that $Q$-a.s.\
 $f(g_+  +  h_0/\tilde  q)  =  f(g_+)  +  f^*(0)  h_0/  \tilde  q$.   By
 Lemma~\ref{lem:cvx-prop}, as  $f$ is strictly convex  on $\R_+^*$, this
 equality implies that $Q$-a.s.\ $h_0/\tilde  q=0$, that is, $C=0$, and,
 using~\eqref{eq:C},                        we                       get
 $\E_\lambda[p\ind_{\{q=0, \,  \tilde   q   >0\}}]=0$.   Since   trivially
 $\E_\lambda[q\ind_{\{q=0, \,    \tilde    q    >0\}}]=0$    and    $p+q=2$
 $\lambda$-a.s.,   we   get    that   $\ind_{\{q=0, \,   \tilde   q>0\}}=0$
 $\lambda$-a.s., and using~\eqref{eq:p-tp=0} (with $q$  instead of $p$) that
 $\ind_{\{q=0\}}= \ind_{\{\tilde q=0\}}$ $\lambda$-a.s..
 This in turns implies that $\lambda$-a.s.:
 \[
   \ind_{\{\tilde p>0, \tilde q=0\}}
   =  \ind_{\{\tilde p>0,  q=0\}}
   =  \ind_{\{p>0, \tilde p>0,  q=0\}}
   =  \ind_{\{p>0,   q=0\}},
 \]
 using   that $\{p=0, q=0\}$ has zero
 $\lambda$-measure for the second equality, and~\eqref{eq:p-tp=0} for the last; and thus,
as $\tilde q$ is $\sigma(\varphi)$-measurable:
 \begin{equation}
   \label{eq:q=qt}
   \frac{h_0}{\tilde q}
   = \frac{\E_\lambda[p \ind_{\{\tilde q =0\}} \, |\, \varphi]}{\tilde q}
   =\frac{\tilde p}{\tilde q} \ind_{\{\tilde  q=0\}}=\infty  \ind_{\{p>0,   q=0\}}
   =\frac{\rd P}{\rd Q}  \ind_{\{ q=0\}}
   \quad\text{ $\lambda$-a.s..}
 \end{equation}

 Since  $\rd P/\rd  Q$  takes values  in $  \R$  under $Q$,  Theorem~6.3
 in~\cite{k:foundations} provides the existence  of a probability kernel
 on $E\times \cb(\bar \R)$ which is a regular version of the probability
 distribution of  $\rd P/\rd  Q$ under  $Q$ conditionally  on $\varphi$.
 The  equality in  Jensen  inequality~\eqref{eq:A1-majo2}  (for the  the
 probability distribution  of $\rd P/\rd  Q$ under $Q$  conditionally on
 $\varphi$)   is   achieved  if   and   only   if  $Q$-a.s.\   we   have
 $\E_Q[\rd P/\rd Q\, |\, \varphi]=\rd P/\rd Q$.
Since  $\ind_{\{q>0\}}= \ind_{\{\tilde q>0\}}$ $\lambda$-a.s., we get
that   $\lambda$-a.s. $g_+=(\rd P/\rd Q)\ind_{\{q>0\}}$.
With~\eqref{eq:q=qt} and~\eqref{eq:P'/Q'}, this gives:
\[
     \frac{\rd P'}{\rd Q'}\circ \varphi
   = \frac{\tilde p}{\tilde q}
   = g_+  + \frac{h_0 }{\tilde q}
   = \frac{\rd P}{\rd Q}   \quad\text{ $\lambda$-a.s..}
 \] 
\end{proof}

\begin{proof}[Proof of Proposition~\ref{prop:markov}]
Since $X$ and $Z$ are independent conditionally on $Y$, it is elementary
to check that:
\[
  P_{(X,Y,Z)}(\rd x, \rd y, \rd z)= P_{(X,Y)}(\rd x , \rd y)
  P_{Z|Y=y}(\rd z).
\]
  Indeed, let $\phi$  be some nonnegative measurable function.
We have
$$
\E_{(X,Y,Z)}[\phi(Z)|X,Y]=\E_{(Y,Z)}[\phi(Z)|Y]
$$ which we
denote by  $\Phi(Y)$.
 Let $\psi$  be some nonnegative measurable function. We get:
 \[
    \E_{(X,Y,Z)}\left[ h(X,Y) \psi(X,Y) \phi(Z)\right]
   =    \E_{(X,Y,Z)}\left[ h(X,Y) \psi(X,Y) \Phi(Y)\right],
 \]
 which gives that $ P_{(X,Y,Z)}(\rd x, \rd y, \rd z)= P_{(X,Y)}(\rd x , \rd y)
  P_{Z|Y=y}(\rd z)$.

Since  $  P_{Z|Y=y}(\rd   z)=  P_{Z'|Y'=y}(\rd  z)$  for   all  $y$,  we
get~\eqref{eq:densite-Markov}.  Then use  Proposition~\ref{prop:inv-gen}
with         $\varphi(x,y,z)=(x,y)$        to         deduce        that
$   D_f\left(P_{(X',    Y',   Z')}   \,\Vert    \,   P_{(X,Y,Z)}\right)=
D_f\left(P_{(X',   Y')}  \,   \Vert   \,   P_{(X,Y)}\right)$  and   then
Theorem~\ref{thm:increase-Df}     with     $\varphi(x,y,z)=(x,z)$     to
get~\eqref{eq:ineg-Df-Markov}:
$$
  D_f\left(P_{(X',  Z')}   \,\Vert    \,   P_{(X,Z)}\right) \leq
D_f\left(P_{(X',   Y')}  \,   \Vert   \,   P_{(X,Y)}\right)
$$ 
\end{proof}


\subsection{Proofs of results in Section~\ref{sec:csiszar_index}}
\label{sec:proofs_copula}

Before giving the proof of Lemma~\ref{lem:interpol-cop=cop} we provide a
technical lemma. We say  that two events $A$ and $B$  are equal a.s.\ if
$\P(A\cap B^c)+\P(A^c \cap  B)=0$, and use similar notation for the
a.s.\ inclusion.  For  a real-valued random
 variable $X$ with cdf $F$,  we denote by $F^{-1}$  the  generalized (c\`ag-l\`ad)  inverse of  $F$ given  by
  $F^{-1}(t)=\inf\{  s\in  \R  \,   \colon\,  F(s)\geq  t\}$,  with  the
  convention that $\inf \emptyset=+\infty $  and $\inf \R=-\infty $.
 In particular, for
  all $s,  t\in \R$, we have:
  \begin{equation}
      \label{eq:F-1}
      F(s)\leq  t \Longleftrightarrow s\leq   F^{-1} (t).
  \end{equation}
\begin{lem}[Level sets]
  \label{lem:X-FX}
  Let $X$ be a real-valued  random variable with cdf $F$. We have that, for  $x\in \R$:
 \begin{equation}
    \label{eq:F-2}
  \{F(X) \leq  F(x)\}=  \{X\leq  x\} 
    \quad\text{a.s.}.
  \end{equation}
\end{lem}

\begin{proof} For $t \in [0,1]$, as $\P(X\in (F^{-1}(t), F^{-1}(t+)))=0$,  we deduce from~\eqref{eq:F-1} that a.s.\ $F(X)\leq t \Longleftrightarrow X\leq  F^{-1}( t)$.
  For all $s\in \R$, we get $s \geq F^{-1} \circ F(s) $ and we deduce that a.s.\
  $F^{-1} \circ F(X)=X$.
We have $\{X\leq  x\}\subset \{F(X) \leq  F(x)\}$ as well as $\P(X\leq
x)=\P(F^{-1} \circ F(X) \leq  x)=\P(F(X) \leq  F(x))$ and
thus~\eqref{eq:F-2} holds. 
\end{proof}

\begin{proof}[Proof of Lemma~\ref{lem:interpol-cop=cop}]
  Let  $X$  be a  real-valued  random  variable. Recall its discrete
  support is denoted $\Delta=\Delta(X)$. Let   $(W_x)_{x\in
    \Delta}$ be a family of  uniform  random
  variable on $[0,1]$  independent of $X$,  and set, with  $F$ the cdf of
  $X$:
\[
  Y=F(X-) +   \sum_{x\in  \Delta} \P(X=x)\, W_x \ind_{\{X=x\}}.
\]
  We recall that $Y\in [F(X-),
  F(X)]$.

Let $x\in \R$. By decomposing according to $X$ belonging or not to $\Delta$, we get
that a.s.:
\begin{align}
  \nonumber
  \{Y\leq
  F(x)\}
  &= \{F(X-) < F(x), \, X\in \Delta\} \cup\{F(X) \leq  F(x),\,
    X\not\in \Delta\}\\
  \nonumber
   &= \{F(X)\leq F(x), \, X\in \Delta\}
    \cup \{F(X) \leq  F(x),\,    X\not\in \Delta\}\\
  &= \{X\leq  x\},
    \label{eq:F-XY}
\end{align}
where  we used that  $W_x>0$ a.s.\ for the
first equality, and~\eqref{eq:F-2} for the last.

This implies  that  $\P(Y\leq  y)=y$ for all $y\in F(\R)$. 
    For $x\in  \Delta$ and $y\in [F(x-), F(x)]$,
  we also have:
\[
  \P(Y\leq    y)=\P(X<x)+    \P\left(X=x,   W_x\leq    \frac{y-F(x-)}{F(x)
      -F(x-)}\right) =y .
\]
Since
$(0,1)\subset  F(\R)\cup_{x\in \Delta}[F(x-), F(x)]
$, we  deduce that  $Y$ is  uniform on  $[0,1]$.

\medskip

This latter part implies that  the random
variables   $U_1, \ldots, U_d$ defined in~ \eqref{eq:def-U1-d} are  uniformly  distributed  over
$[0,1]$. Applying~\eqref{eq:F-XY} to
$(Z_i,U_i)$ (and $T_{i,x}$) instead of $(X,Y)$ (and $W_x$), we
get that~\eqref{eq:definition_copula_vector} holds with $C$ the cdf of
$(U_1, \ldots, U_d)$, 
 that is, $C$ is a copula for $Z$.
\end{proof}

\begin{proof}[Proof of Theorem~\ref{thm:interpolating_minimise_fdiv}]
For simplicity we write  $F_i$ for the cdf $F_{X_i}$ and $G_j$ for the
cdf $F_{Y_j}$, where $X=(X_1, \ldots, X_{d_X})$ and $Y=(Y_1, \ldots,
Y_{d_Y})$. 

We   first  prove~{\it \ref{it:copu-majoD}}.    We  consider   the  measurable
functions    $h  =    (F_1^{-1},   \dots,    F_{d_X}^{-1})$   and
$g=(G_1^{-1},  \dots ,  G_{d_Y}^{-1})$.   Let  
$(U,V)$  be a  random  variable with  cdf $C_{(X,Y)}$, and notice  that $C_X$ and $C_Y$  are the
cdf of $X $ and of $Y$.  Since $(h(U), g(V))$ is distributed as
$(X,Y)$, see~\eqref{eq:F-2}, we deduce from Corollary~\ref{cor:increase_Sf} that for $f\in \F$:
\[
    D_f(C_X \otimes C_Y \Vert C_{(X,Y)}) = S_f(U,V) \geq S_f(h(U),
    g(V)) = S_f(X,Y). 
\]

We now  prove~{\it \ref{it:copu-equalD}}.   By  assumption,  the  random  variables 
$$
T=(T_{i,x})_{i\in   \{1,   \ldots,   d_X\},  x\in   \Delta(X_i)} \text{  and }
W=(W_{j,y})_{j\in   \{1,   \ldots,   d_Y\},  y\in   \Delta(Y_j)},
$$
which appear in~\eqref{eq:def-U}
and~\eqref{eq:def-V},    are
independent       and       independent      of       $(X,Y)$.        By
Lemma~\ref{lem:interpol-cop=cop},  we  recall that  the  one-dimensional
marginals  of  the  random  vectors  $U=(U_1, \ldots,  U_{d_X})$  and
$V=( V_1, \ldots,  V_{d_Y})$ are uniformly distributed  over $[0,1]$ and
that the cdf of  $(U,V)$, of $U$ and of $V$  are an interpolating copula
for $(X,Y)$, of $X$ and of $Y$.

  We write $\tilde T=(\tilde T_{i,x})_{i\in  \{1, \ldots,  d_X\}, x \in  \Delta_{X_i} }
   $ with $\tilde  T_{i,x}= T_{i,x} \ind_{\{X_i=x\}}$, and we similarly
   define $\tilde W$. Notice there 
   exists a measurable injection $\varphi$ (resp.\ $\psi$) from $[0, 1]^{d_X}$ to $\bar \R^{d_X} \times [0, 1]^{d_X}$ (resp.\
   from $[0, 1]^{d_Y}$ to $\bar \R^{d_Y} \times [0, 1]^{d_Y}$) such that $\varphi(U)=(X ,\tilde T)$
   (resp.\ $\psi(V)=(Y, \tilde W)$). Since $\varphi$ and $\psi$ are bi-measurable, see~\cite{purves1966bimeasurable}, we deduce from
   Proposition~\ref{prop:invariance_csiszar_index}~{\it \ref{enum:bijection_invariance}} that:
$$
S_f(U,V)=S_f((X, , \tilde T), (Y, \tilde W)).
$$
Using that:
   \begin{align*}
      P_{(X,Y, \tilde T, \tilde W)} (\rd x, \rd y,
     \rd t ,\rd w)
     &=P_{(X,Y)}(\rd x, \rd y) P_{\tilde T|X=x} (\rd t) 
   P_{\tilde        W|Y=y} (\rd w) ,\\
     P_{(X, \tilde T)}\otimes P_{(Y, \tilde W)} (\rd x, \rd y,
     \rd t , \rd w)
     &=P_X\otimes P_Y (\rd x, \rd y) P_{\tilde T|X=x} (\rd t) 
   P_{\tilde      W|Y=y} (\rd w),
   \end{align*}
   we deduce from Proposition~\ref{prop:invariance_csiszar_index}~{\it \ref{it:Sf-Markov}}
 that:
 $S_f((X, \tilde T), (Y, \tilde W))=S_f(X,Y)$. This concludes the
proof.

Point~{\it \ref{it:copu-cb}} is a direct consequence of~{\it \ref{it:copu-majoD}}
and~{\it \ref{it:copu-equalD}}.
\end{proof}

\end{document}